\documentclass[11pt]{amsart}

\usepackage[all, cmtip]{xy}
\usepackage{amsmath,amsfonts,amssymb,amsthm,pict2e,enumerate,chngcntr,caption,tikz}
\usepackage[T1]{fontenc}
\usepackage{longtable}

\theoremstyle{definition}
\newtheorem{defn}{Definition}[section]

\newtheorem{ex}[defn]{Example}

\theoremstyle{plain}
\newtheorem{lemma}[defn]{Lemma}
\newtheorem{theorem}[defn]{Theorem}
\newtheorem{proposition}[defn]{Proposition}
\newtheorem{corollary}[defn]{Corollary}

\usepackage{setspace}

\usepackage[margin=1in,footskip=0.25in]{geometry}

\newcommand{\rank}{{\mathrm {rank}}}
\newcommand{\Pic}{{\text {Pic}}}

\newcommand{\Gr}{\mathrm{Gr}}

\newcommand{\Hom}{\mathrm{Hom}\,}

\counterwithin{table}{section}
\begin{document}
\title[Degenerations and Landau-Ginzburg models]{Toric Degenerations and the Laurent polynomials related to Givental's Landau-Ginzburg models}

\author[C. F. Doran]{Charles F. Doran}
\address{Department of Mathematical and Statistical Sciences, 632 CAB, University of Alberta, Edmonton, Alberta T6G 2G1, Canada}
\email{charles.doran@ualberta.ca}

\author[A. Harder]{Andrew Harder}
\address{Department of Mathematical and Statistical Sciences, 632 CAB, University of Alberta, Edmonton, Alberta T6G 2G1, Canada}
\email{aharder@ualberta.ca}
\subjclass[2010]{14J28 (primary), 14D05, 14J32 (secondary)}

\date{\today}

\begin{abstract}
For an appropriate class of Fano complete intersections in toric varieties, we prove that there is a concrete relationship between degenerations to specific toric subvarieties and expressions for Givental's Landau-Ginzburg models as Laurent polynomials. As a result, we show that Fano varieties presented as complete intersections in partial flag manifolds admit degenerations to Gorenstein toric weak Fano varieties, and their Givental Landau-Ginzburg models can be expressed as corresponding Laurent polynomials.

We also use this to show that all of the Laurent polynomials obtained by Coates, Kasprzyk and Prince by the so called Przyjalkowski method \cite{ckp} correspond to toric degenerations of the corresponding Fano variety. We discuss applications to geometric transitions of Calabi-Yau varieties.
\end{abstract}
\maketitle
\section{Introduction}

Mirror symmetry for Fano varieties predicts that the mirror of a Fano variety $X$ is given by a quasi-projective variety $X^\vee$ equipped with a regular function $w:X^\vee \rightarrow \mathbb{A}^1$ (with appropriate choices of symplectic and complex structure on both $X$ and $X^\vee$) which satisfies certain conditions. In particual, homological mirror symmetry implies that there is a relationship between the bounded derived category of $X$ and the Fukaya-Seidel category of $(X^\vee,w)$, or conversely, the Fukaya category of $X$ is related to the derived category of singularities of $(X^\vee,w)$ (see, for instance, \cite{kkp1,kkp2} for details). The pair $(X^\vee,w)$ can be viewed as a family of varieties over $\mathbb{A}^1$. From a more classical point of view, mirror symmetry predicts that the periods of this family at infinity should be related to the Gromov-Witten invariants of $X$ \cite{ehx}.

As an example, if $X$ is a smooth $n$-dimensional toric Fano variety, then there should be some copy of $(\mathbb{C}^\times)^n$ contained in $X^\vee$ so that on this torus, $w$ is expressed as a Laurent polynomial
$$
w : (\mathbb{C}^\times)^n \rightarrow \mathbb{C}
$$
with Newton polytope equal to the polytope whose face fan from which $X$ itself is constructed. 

Galkin and Usnich (Problem 44, \cite{gu}) and Przyjalkowski (Optimistic picture 38, \cite{prz}) suggest that for each birational map
$$
\phi: (\mathbb{C}^\times)^n \dashrightarrow X^\vee
$$
so that the Newton polytope of $\phi^*w$ is $\Delta$, there is a degeneration of the Fano variety $X$ to $X_\Delta$. It is possible that $X^\vee$ is covered (away from a subset of codimension 2) by tori $(\mathbb{C}^\times)^n$ corresponding to toric varieties to which $X$ degenerates, and these charts are related by a generalized type of cluster mutation. Our main result (Theorem \ref{thm:main}) is very much in the spirit of this suggestion.

For $X$ a complete intersection in a toric variety, Givental \cite{Givental} provided a method of computing the Landau-Ginzburg model of $X$. This Landau-Ginzburg model is presented as complete intersection in $(\mathbb{C}^\times)^n$ which we call $X^\vee$ equipped with a function $w$. We call the pair $(X^\vee,w)$ obtained by Givental's method the {\it Givental Landau-Ginzburg model of $X$}.

In Section \ref{sect:general}, we introduce certain types of embedded toric degenerations of Fano complete intersections in toric varieties which we call {\it amenable toric degenerations}, and prove that they correspond to Laurent polynomial models of Givental's Landau-Ginzburg models.

\begin{theorem}[Theorem \ref{thm:main}]
Let $X$ be a complete intersection Fano variety in a toric variety $Y$. Let $X \rightsquigarrow X_\Sigma$ be an amenable toric degeneration of $X$, then the Givental Landau-Ginzburg model of $X$ can be expressed as a Laurent polynomial with Newton polytope equal to the convex hull of the rays generating the 1-dimensional strata of $\Sigma$.
\end{theorem}

In the case where $X$ is a smooth complete intersection in a weighted projective space, Przyjalkowski showed that there is a birational map $\phi:(\mathbb{C}^\times)^m \dashrightarrow X^\vee$ to the Givental Landau-Ginzburg model of $X$ so that $\phi^*w$ is a Laurent polynomial.  In \cite{ilp11}, Ilten, Lewis and Przyjalkowski have shown that there is a toric variety $X_\Delta$ expressed as a binomial complete intersection in the ambient weighted projective space so that the complete intersection $X$ admits a flat degeneration to $X_\Delta$.

Theorem \ref{thm:main} generalizes both the method of Przyjalkowski in \cite{prz}, and its subsequent generalization by Coates, Kasprzyk and Prince in \cite{ckp}. Theorem \ref{thm:main} shows that there are toric degenerations corresponding to all of the Laurent polynomials associated to Fano fourfolds obtained in \cite{ckp}, and that the Laurent polynomials are the toric polytopes of the associated degenerations. 

Using the toric degeneration techniques of \cite{gs} and \cite{bcfkvs1}, Przyjalkowski and Shramov have defined Givental Landau-Ginzburg models associated to complete intersection Fano varieties in partial flag varieties. They have shown that the  Givental Landau-Ginzburg models of complete intersections in Grassmannians $\Gr(2,n)$ can be expressed as Laurent polynomials. We provide an alternate approach to this question using Theorem \ref{thm:main}. This provides toric degenerations for most complete intersection Fano varieties in partial flag manifolds, and shows that we may express their Givental Landau-Ginzburg models as Laurent polynomials.

\begin{theorem}[Theorem \ref{thm:flags}]
Many Fano complete intersections in partial flag manifolds admit degenerations to toric weak Fano varieties $X_\Delta$ with at worst Gorenstein singularities and the corresponding Givental Landau-Ginzburg models may be expressed as Laurent polynomials with Newton polytope $\Delta$.
\end{theorem}

Of course, the word ``many'' will be explained in detail in Section \ref{sect:flag}, but as an example, this theorem encompasses all complete intersections in Grassmannians.

\subsection{Organization}
This paper will be organized as follows. In Section \ref{sect:general}, we will recall facts about toric varieties and use them to prove Theorem \ref{thm:main}. In Section \ref{sect:flag}, we will apply the results of Section \ref{sect:general} to exhibit toric degenerations of complete intersections Fano varieties in partial flag manifolds and show that their Givental Landau-Ginzburg models admit presentations as Laurent polynomials. In Section \ref{sect:apply}, we shall comment on further applications to the Przyjalkowski method of \cite{ckp} and how our method seems to relate to the construction of toric geometric transitions as studied by Mavlyutov \cite{mav2} and Fredrickson \cite{fred}.

\subsection*{Acknowledgements}
The second author would like to thank V.Przyjalkowski for many insightful and encouraging conversations and for suggesting the application to complete intersections in partial flag varieties. C. Doran was supported the National Sciences and Engineering Research Council, the Pacific Institute for Mathematical Sciences and the McCalla professorship at the University of Alberta. A.Harder was supported by a National Sciences and Engineering Research Council Doctoral Post-Graduate Scholarship.
\section{General results}
\label{sect:general}
Here we describe the relationship between degenerations of complete intersections in toric Fano varieties with nef anticanonical divisor and their Landau-Ginzburg models. We will begin with a rapid recollection of some facts about toric varieties. A general reference for all of these facts is \cite{cls}.

\subsection{Toric facts and notation}

Throughout, we will use the convention that $M$ denotes a lattice of rank $n$, and $N$ will be $\Hom(M,\mathbb{Z})$. We denote the natural bilinear pairing between $N$ and $M$ by 
$$
\langle -, - \rangle: N \times M \rightarrow \mathbb{Z}.
$$
The symbol $\Sigma$ will denote a complete fan in $M \otimes_\mathbb{Z} \mathbb{R}$, and $X_\Sigma$ or $Y_\Sigma$ will be used to denote the toric variety associated to the fan $\Sigma$. We will let $\Delta$ be a convex polytope in $M\otimes_\mathbb{Z} \mathbb{R}$ with all vertices of $\Delta$ at points in $M$, which contains the origin in its interior. 

We will let $\Sigma_\Delta$ be the fan over the faces of the polytope $\Delta$, and we will also denote the toric variety $X_{\Sigma_\Delta}$ by $X_\Delta$. If $\Delta$ is an integral convex polytope, then we will let $\Delta[n]$ be the set of dimension $n$ strata of $\Delta$. In particular, denote by $\Delta[0]$ the vertices of $\Delta$. We will abuse notation slightly and let $\Sigma[1]$ be the set of primitive ray generators of the fan $\Sigma$. Similarly, if $C$ is a cone in $\Sigma$, then $C[1]$ will denote the set of primitive ray generators of $C$.

There is a bijection between primitive ray generators of $\Sigma$ and the torus invariant Weil divisors on $X_\Sigma$.

If $X_\Sigma$ is a toric variety, then $X_\Sigma$ has a Cox homogeneous coordinate ring which is graded by $G_{\Sigma} = \Hom(\mathrm{A}_{n-1}(X_\Sigma),\mathbb{C}^\times)$. There is a short exact sequence
$$
0 \rightarrow N \rightarrow \mathbb{Z}^{\Sigma[1]} \rightarrow \mathrm{A}_{n-1}(X_\Sigma) \xrightarrow{g} 0
$$
where the first map assigns to a point $u \in N$ the vector
$$
(\langle n, \rho \rangle)_{\rho \in \Sigma[1]}.
$$
Elements of $\mathbb{Z}^{\Sigma[1]}$ are in bijection with torus invariant Weil divisors and the map $g$ assigns to a torus invariant Weil divisor on $X_\Sigma$ its class in the Chow group $\mathrm{A}_{n-1}(X_\Sigma)$. 

Applying the functor $\Hom(-,\mathbb{C}^\times)$ to the above short exact sequence, we get a sequence
$$
1 \rightarrow G_\Sigma \rightarrow (\mathbb{C}^\times)^{\Sigma[1]} \rightarrow T_M \rightarrow 1
$$
where $T_M = M \otimes_\mathbb{Z} \mathbb{C}^\times$. Let $x_\rho$ be a standard basis of rational functions on $(\mathbb{C}^\times)^{\Sigma[1]}$. There is a partial compactification of $(\mathbb{C}^\times)^{\Sigma[1]}$, which we may call $V_\Sigma$ 
$$
(\mathbb{C}^\times)^{\Sigma[1]} \subseteq V_\Sigma \subseteq \mathbb{C}^{\Sigma[1]}.
$$
so that there is an induced action of $G_\Sigma$ on $V_\Sigma$ and linearizing line bundle so that the categorical quotient $V_\Sigma// G_\Sigma$ is the toric variety $X_\Sigma$. Since we have assumed that $\Sigma$ is complete, the homogeneous coordinate ring of $X_\Sigma$ is $\mathbb{C}[\{x_\rho\}_{\rho \in \Sigma[1]}]$ equipped with the grading given by the action of $G_\Sigma$. We will abuse notation and say that a subvariety of $X_\Sigma$ is a complete intersection in $X_\Sigma$ if it corresponds to a quotient of a complete intersection in $V_\Sigma$. 

The sublocus of $X_\Sigma$ corresponding to $D_\rho = \{x_\rho = 0 \}$ is exactly the torus invariant divisors associated to the ray generator $\rho$. Despite being given by the vanishing of a function in the homogeneous coordinate ring, these divisors need not be Cartier. A torus invariant divisor $D =\sum_{\rho \in \Sigma[1]} a_\rho D_\rho$ is Cartier if and only if there is some piecewise linear function $\varphi$ on $M \otimes_\mathbb{Z} \mathbb{R}$ which is linear on the cones of $\Sigma$, which takes integral values on $M$. If $\phi$ is upper convex then the divisor $D$ is nef.

The canonical divisor of a toric variety $X_\Sigma$ is the divisor $K_{X_\Sigma} = -\sum_{\rho \in \Sigma[1]} D_\rho$. A toric variety $X_\Sigma$ is called $\mathbb{Q}$-Gorenstein if its canonical divisor is $\mathbb{Q}$-Cartier, and Gorenstein if its canonical divisor is Cartier. In the future, we will be concerned solely with $\mathbb{Q}$-Gorenstein toric varieties.

A {\it nef partition} of $\Sigma$ will be a partition of $\Sigma[1]$ into sets $E_1,\dots, E_{k+1}$ so that there exist integral upper convex piecewise linear functions $\varphi_1,\dots, \varphi_{k+1}$ so that 
$$
\varphi_i(E_j) \geq \delta_{ij}.
$$
This means that for each maximal cone $C$ in the fan $\Sigma$, there is some vector $u_C \in N$ so that 
$$
\varphi_i(v) = \mathrm{max} \{\langle u_C, v \rangle \}_{C \in \Sigma}
$$
A {\it $\mathbb{Q}$-nef partition} will be a partition of $\Sigma[1]$ exactly as above, except we no longer require that the functions $\varphi_i$ be integral, but only that they take rational values on $u\in M$. This is equivalent to the fact that each $\varphi_i$ is determined by a vector $u_C \in N \otimes_\mathbb{Z} \mathbb{Q}$ for each maximal cone $C$ of $\Sigma$. 

The divisors determined by a $\mathbb{Q}$-nef partition are $\mathbb{Q}$-Cartier. Note that the existence of a $\mathbb{Q}$-nef partition implies that $X_\Sigma$ is $\mathbb{Q}$-Gorenstein and the existence of a nef partition implies that $X_\Sigma$ is Gorenstein.

\subsection{Amenable collections of vectors}

We begin by letting $X$ be a complete intersection in a toric variety $Y_\Sigma$ of the following type.
\begin{defn}
We say that $X$ is a {\it quasi-Fano complete intersection} in $Y_\Sigma$ if there are divisors $Z_1,\dots, Z_k$ defined by homogeneous equations $f_i$ in the homogeneous coordinate ring of $Y_\Sigma$ so that $(f_1,\dots,f_k)$ forms a regular sequence in $\mathbb{C}[\{x_\rho\}_{\rho \in \Sigma[1]}]$, and there is a $\mathbb{Q}$-nef partition $E_1,\dots,E_{k+1}$ so that 
$$
Z_i \sim \sum_{\rho \in E_i} D_\rho.
$$
\end{defn}

We now define the central object of study in this paper. Fix a $\mathbb{Q}$-nef partition of $\Sigma$ as $E_1,\dots, E_{k+1}$ and let $X$ be a corresponding quasi-Fano complete intersection. 

\begin{defn}
An {\it amenable collection of vectors} subordinate to a $\mathbb{Q}$-nef partition $E_1,\dots,E_{k+1}$ is a collection $V = \{v_1,\dots,v_k\}$ of vectors satisfying the following three conditions
\begin{enumerate}
\item For each $i$, we have $\langle v_i, E_ i \rangle = -1$
\item For each $j$ so that $k+1 \geq j \geq i+1$, we have $\langle v_i, E_j \rangle \geq 0$.
\item For each $j$ so that $0 \leq j \leq i-1$, we have $\langle v_i, E_j \rangle = 0$.
\end{enumerate}
\end{defn}
Note that this definition depends very strongly upon the order of $E_1,\dots, E_k$. Now let us show that an amenable collection of vectors may be extended to a basis of $N$.
\begin{proposition}
\label{prop:basis}
An amenable set of vectors $v_i$ spans a saturated subspace of $N$. In particular, there is a basis of $N$ containing $v_1,\dots, v_{k}$. 
\end{proposition}
\begin{proof}
First of all, it is clear that $k \leq \rank(M)$, or else $X$ would be empty. Now for each $E_i$, choose some $\rho_i \in E_i$. We then have a map
$$
\eta: \mathrm{span}_\mathbb{Z}(v_1,\dots, v_k) \rightarrow \mathbb{Z}^k
$$
determined by 
$$
\eta(v) = (\langle v,\rho_1 \rangle, \dots, \langle v , \rho_k \rangle).
$$
Which, when expressed in terms of the basis $v_1,\dots,v_k$ is upper diagonal with $(-1)$ in each diagonal position. Thus $\eta$ is an isomorphism. If $\mathrm{span}_\mathbb{Z}(v_1,\dots,v_k)$ were not saturated, then there is some $v \in \mathrm{span}_{\mathbb{Q}}(v_1,\dots,v_k) \cap N$ which is not in $\mathrm{span}_\mathbb{Z}(v_1,\dots,v_k)$. But then $\eta(v)$ could not lie in $\mathbb{Z}^k$, which is absurd, since $v\in N$ and $\rho_i$ are elements of $M$ and by definition $\langle v, \rho_i \rangle \in \mathbb{Z}$.

Thus the embedding $\mathrm{span}_\mathbb{Z}(v_1,\dots,v_k) \hookrightarrow N$ is primitive and there is a complementary set of vectors $v_{k+1},\dots,v_n$ so that $v_1,\dots,v_n$ spans $N$ over $\mathbb{Z}$.
\end{proof}
Now we will proceed to show that amenable collections of vectors lead naturally to a specific class of complete intersection toric subvarieties of $Y_\Sigma$.
\subsection{Toric degenerations}

Now let us define a toric variety associated to an amenable collection of vectors subordinate to a $\mathbb{Q}$-nef partition of a fan $\Sigma$.

\begin{defn}
Let $V$ be an amenable collection of vectors subordinate to a $\mathbb{Q}$-nef partition of a fan $\Sigma$. Let $M_V$ be the subspace of $M \otimes_\mathbb{Z} \mathbb{R}$ composed of elements which satisfy $\langle v_i, u \rangle = 0$ for each $v_i \in V$. Define the fan $\Sigma_V$ to be the fan in $M_V$ induced by the fan $\Sigma$.
\end{defn}

We may now look at the subvariety of $Y_\Sigma$ determined by the equations  
$$
\prod_{\rho \in E_i} x_\rho - \prod_{\rho \notin E_i} x_\rho^{\langle v_i, \rho\rangle } = 0
$$
in its homogeneous coordinate ring. Note that if $X$ is a quasi-Fano complete intersection in $Y_\Sigma$ determined by a $\mathbb{Q}$-nef partition, then if the variety determined by the equations above is a degeneration of $X$.

Following \cite{fish}, we introduce a definition:
\begin{defn}
An integral square matrix is called mixed if each row contains both positive and negative entries. A $k \times m$ matrix is called mixed dominating if there is no square submatrix which is mixed.
\end{defn}
Mavlyutov (\cite{mav2} Corollary 8.3) packages Corollaries 2.4 and 2.10 of \cite{fish} into the following convenient form. If $l$ is an integral vector in a lattice expressed in terms of a fixed basis, as $l=(t_1,\dots, t_n)$ then we define monomials $x^{(l_-)}$ and $x_{(l_+)}$ to be
$$
x^{(l_+)} = \prod_{t_i > 0} x_i^{t_i}, \, \,  \, \, x^{(l_-)} = \prod_{t_i < 0} x_i^{t_i}.
$$
\begin{proposition}
Let $L = \bigoplus_{i=1}^k \mathbb{Z}l_i$ be a saturated sublattice of $\mathbb{Z}^m$ so that $L \cap \mathbb{N}^m = \{0\}$. Assume that the matrix with rows $l_1,\dots, l_k$ is mixed dominating, then the set of polynomials $(x^{(l_i)_+} - x^{(l_i)_-})$ for $i = 1,\dots, k$ forms a regular sequence in $\mathbb{C}[x_1,\dots,x_m]$
\end{proposition}
\begin{proposition}
\label{prop:Mav}
If $X$ is a Fano complete interesection in a toric variety $Y_\Sigma$, and $V$ is an amenable collection of vectors associated to $X$, then $V$ determine a degeneration of $X$ to a complete intersection toric variety in the homogenous coordinate ring of $Y_\Sigma$ cut out by equations
$$
\prod_{\rho \in E_i} x_\rho - \prod_{\rho \notin E_i} x_\rho^{\langle v_i, \rho\rangle } = 0.
$$
\end{proposition}
\begin{proof}
The definition is clear, but what needs to be shown is that the subvariety of $Y_\Sigma$ determined by these equations is a complete intersection. This is equivalent to the fact that the equations given in the statement of the proposition form a regular sequence. To check this, we prove that the conditions of Proposition \ref{prop:Mav} hold. The relevant matrix is the matrix with rows
$$
(\langle v_i, \rho \rangle)_{\rho \in \Sigma[1]}.
$$
Call this matrix $T$. If we choose some $\rho_j \in E_j$ for each $1 \leq i \leq k$, then the maximal submatrix $(\langle v_i, \rho_j \rangle)$ is upper triangular with $(-1)$s on the diagonal. Thus the rows of $T$ form a saturated sublattice of $\mathbb{Z}^m$. 

Now let us choose any square submatrix of $T$, or in other words, choose a set $U$ of $\ell$ vertices of $\Sigma[1]$ and a set $V$ of $\ell$ vectors $v_i$. Then we must show that the matrix $S = (\langle v_i, \rho\rangle)_{v_i \in V, \rho \in U}$ has a row without both positive and negative entries. If the row $(\langle v_i,\rho\rangle)_{\rho \in \Sigma[1]}$ contains negative entries for $V_0 = \{v_{i_1},\dots,v_{i_m}\} \subseteq V$, then in particular for each $v_{i_j}$ there is some $\rho \in U$ contained in $E_{i_j}$. If $V_0 = V$, then it follows that for each $\rho$ is contained in a distinct $E_{i_1},\dots, E_{i_m}$. If $i_1$ is the smallest such integer and $\rho_1$ is the corresponding element of $U$, then $\langle v_{i}, \rho \rangle \leq 0$ for all $v_i \in V$, since $\langle v_i, E_j \rangle = 0$ for all $j < i$. Thus the corresponding row contains no positive entries.

Finally, if $L \cap \mathbb{N}^m$ is nonzero then there is some $v_i$ so that $\langle v_i, \rho \rangle \geq 0$ for all vertices $\rho$ of $\Sigma[1]$. If there were such a $v_i$, then all points $\rho$ of $\Sigma[1]$ would be contained in the positive half-space determined by $v_i$, contradicting the fact that $\Sigma$ is a complete fan with each cone strictly convex.

Thus applying Proposition \ref{prop:Mav}, the equations in the proposition above determine a complete intersection in the homogeneous coordinate ring of $Y_\Sigma$.
\end{proof}

\begin{proposition}
\label{prop:fan}
The subvariety of $X_{\Sigma_V}$ of $Y_\Sigma$ corresponds to the complete intersection in the coordinate ring of $Y_\Sigma$ cut out by equations
$$
\prod_{\rho \in E_i} x_\rho - \prod_{\rho \notin E_i} x_\rho^{\langle v_i, \rho\rangle } = 0.
$$ 
for $1 \leq i \leq k$.
\end{proposition}
\begin{proof}
We recall that there is an exact sequence
$$
0 \rightarrow N \xrightarrow{g} \mathbb{Z}^{\Sigma[1]} \rightarrow \mathrm{A}_{n-1}(X_\Sigma) \rightarrow 0.
$$
Here $g$ is the map which sends a point $v \in N$ to the point 
$$
(\langle v, \rho \rangle)_{\rho \in \Sigma[1]} \in \mathbb{Z}^{\Sigma[1]}
$$
Applying the functor $\Hom( - ,\mathbb{C}^\times)$ to this exact sequence, we obtain a dual exact sequence
$$
 0 \rightarrow G_\Sigma \rightarrow (\mathbb{C}^\times)^{\Sigma[1]} \xrightarrow{g^*} M \otimes_\mathbb{Z} \mathbb{C}^\times \rightarrow 0
$$
where $G_\Sigma = \Hom(\mathrm{A}_{n-1}(X_\Sigma),\mathbb{C}^\times)$. For an appropriate choice of basis $(x_\rho)_{\rho \in \Sigma[1]}$ the induced action of $G_\Sigma$ on $(\mathbb{C}^\times)^{\Sigma[1]} \subseteq V_\Sigma$ determines the $G_\Sigma$-action on the homogeneous coordinate ring of $X_\Sigma$. 

The equations by which we have defined $X_\Sigma$ can be written on the torus $(\mathbb{C}^\times)^{\Sigma[1]}$ as 
$$
\prod_{\rho \in \Sigma[1]} x_\rho^{\langle v_i, \rho \rangle} = 1
$$
for $1 \leq i \leq k$. But this corresponds exactly to the pullback of the locus $\langle v_i, - \rangle = 0$ in $M$, which is simply the subspace $M_V$. Thus, in the homogeneous coordinate ring of $Y_{\Sigma}$, the toric subvariety $X_{\Sigma_V}$ is cut out by the equations given in the proposition.
\end{proof}

Thus the amenable collection $V$ determines a degeneration of a quasi-Fano complete intersection in $Y_\Sigma$ to a toric variety $X_{\Sigma_V}$ where $\Sigma_V$ is the fan obtained by intersecting $\Sigma$ with the subspace of $M$ orthogonal to elements of $V$.

\begin{defn}
A toric degeneration $X \rightsquigarrow X_\Sigma$ of a quasi-Fano complete intersection determined by an amenable collection of vectors subordinate to a $\mathbb{Q}$-nef partition $E_1,\dots, E_{k+1}$ is called an amenable toric degeneration of $X$ subordinate to the $\mathbb{Q}$-nef partition $E_{1},\dots, E_{k+1}$. 
\end{defn}

Now we define a polytope depending upon the amenable collection of vectors $V = \{v_1,\dots, v_{k}\}$.
\begin{defn}
Let $V$ be an amenable collection of vectors subordinate to a $\mathbb{Q}$-nef partition $E_1,\dots, E_{k+1}$ equipped with rational convex $\Sigma$-piecewise linear functions $\varphi_1,\dots,\varphi_{k+1}$. Then we define $\Delta_V$ to be the polytope defined by $\rho \in M \otimes \mathbb{R}$ satisfying equations
\begin{align*}
\langle v_i , \rho \rangle &= 0 \text{ for } 1 \leq i \leq k \\
\varphi_{k+1}(\rho) & \leq 1.
\end{align*}
\end{defn}
This polyhedron is convex. We will refer to the subspace of $M \otimes_\mathbb{Z} \mathbb{R}$ satisfying $\langle v_i, \rho \rangle = 0$ for $1 \leq i \leq k$ for $V = \{v_1,\dots, v_k\}$ an amenable collection of vectors as $M_V$. 

It is first of all, important to show that $\Delta_V$ is precisely the polytope whose vertices are the generating rays of the fan $M_V \cap \Sigma_\Delta$.
\begin{lemma}
\label{lemma:Dk+1}
Let $C$ be a sub-cone of $\Sigma$ so that $C \cap M_V$ is nonempty, then there is a vertex of $C$ which is contained in $E_{k+1}$.
\end{lemma}
\begin{proof}
Let $p$ be an element of $(C \cap M_V) \cap M$ and choose a set of vectors $U = \{u_1,\dots, u_m\}$ contained in generating set of the 1-dimensional strata of $C$ so that $p$ is a strictly positive rational linear combination of a set of vectors in $U$. Let $j$ be the largest integer so that $E_j \cap U\neq \emptyset$ and $j \neq k+1$. If no such integer exists, then $U \subseteq E_{k+1}$ and we are done. If not, we have that $\langle v_j, u \rangle = 0$ or $(-1)$ for all $u \in U \setminus E_{k+1}$, since $\langle v_j, E_i \rangle =0$ for $i < j$. If $p' = \sum_{u_i\in U \setminus E_{k+1}} a_i u_i$ for positive numbers $a_i$, then $\langle v_j, p' \rangle = -\sum_{u_i \in E_j}a_i < 0$, since $E_j \cap U$ is nonempty. Thus, since $p = p'+ \sum_{u_i \in U \cap E_{k+1}}a_iu_i$ and $\langle v_j,p \rangle = 0$, we must have $U \cap E_{k+1}$ nonempty.
\end{proof}

The following proposition holds in the case where $E_1,\dots, E_{k+1}$ is any $\mathbb{Q}$-nef partition of $\Delta$ and $V$ is amenable collection of vectors subordinate to this nef partition.

\begin{proposition}
\label{prop:phi1}
Let $C$ be a minimal sub-cone of $\Sigma$ so that $C \cap M_V$ is 1-dimensional, then there is some point $\rho$ in $(C\cap M_V) \cap M$ so that $\varphi_{k+1}(\rho) = 1$.
\end{proposition}
\begin{proof}
By Lemma \ref{lemma:Dk+1}, we may deduce that the set $C[1]$ of primitive integral ray generators of $C$ must contain an element of $E_{k+1}$. We must show that there is some $p \in E_{k+1}$ and $u_1,\dots, u_m \in C[1] \setminus (C[1] \cap E_{k+1})$ so that $\rho = (p + \sum_{i=1}^mn_iu_i) \in M_V$, $n_i >0$. Then since $u_i \in \Sigma[1] \setminus E_{k+1}$ and $\varphi_{k+1}$ is linear on $C$, we must have $\varphi_{k+1}(\rho) = 1$.

Let us take $p$ in $E_{k+1} \cap C[1]$. Then assume $\langle v_j, p \rangle = n_j > 0$. If there is no $u \in C[1]$ so that $u \in E_j$ then $\langle v_j, u \rangle \geq 0$ for all $u \in C[1]$. The subset $\langle v_j, u \rangle =0$ contains $C \cap M_V$ by definition and must be a sub-stratum of $C$ since $C$ is a convex rational cone. By minimality of $C$, we must then have $\langle v_j, u \rangle = 0$ for all $u \in C[1]$. In particular, $n_j =0$ for all $j$ so that $E_j \cap C[1] = \emptyset$.

Now we will fix $u_j \in E_j$ for each $E_j$ so that $E_j \cap C[1] \neq \emptyset$. We know that $\langle v_j, u_j \rangle = -1$. Take the largest $j$ so that $E_j \cap C[1] \neq \emptyset$. Then $\langle v_j, p + n_ju_j \rangle = 0$. Now we have that $p + n_ju_j$ is orthogonal to $v_j$ and since it is a positive sum of elements in $C[1]$, it is contained in $C$. Let $j'$ be the next smallest integer so that $E_{j'} \cap C[1] \neq \emptyset$. Then let $\langle v_{j'}, p + n_ju_j \rangle = s_{j'}$. This is a non-negative integer since $p$ and $u_j$ are not contained in $E_{j'}$. We now have $\langle v_{j'},p + n_ju_j + s_{j'}u_{j'}\rangle = 0$ and $\langle v_j, p + n_ju_j + s_{j'}u_{j'} \rangle = s_{j'}\langle v_j, u_{j'} \rangle$ which is zero by the condition that $\langle v_j ,\Delta_i \rangle =0$ for $i < j$.

We may now sequentially add positive multiples of each $u_\ell$ for $E_\ell \cap C[1] \neq \emptyset$ in the same way until the resulting sum $\rho$ is orthogonal to $v_1,\dots, v_{k}$. Thus we obtain $\rho \in C$ which lies in $M_V$ and satisfies $\varphi_{k+1}(\rho) = 1$ by arguments presented in the first two paragraphs of this proof.
\end{proof}

We may make an even stronger claim if we make further assumptions on the divisors associated to $E_i$.

Recall that we have been assuming that the Weil divisors $D_{i} = \sum_{\rho\in E_{i}} D_\rho$ are $\mathbb{Q}$-Cartier, or in other words, there are rational convex $\Sigma$-linear functions $\varphi_i$ and for $\rho \in D_j$, $\varphi_i(\rho) = \delta_{ij}$. We can make stronger statements about the relationship between the fan $\Sigma_V$ and the polytope $\Delta_V$.

\begin{proposition}
\label{prop:primitive}
The point $\rho$ in the proof of Proposition \ref{prop:phi1} is a primitive lattice point under either of the following two conditions:
\begin{enumerate}
\item The divisor $D_{k+1}$ is Cartier, or
\item All divisors $D_{i}$ for $1 \leq i \leq k$ are Cartier.
\end{enumerate}
\end{proposition}
\begin{proof}
For (1), If $D_{k+1}$ is Cartier, then the function $\varphi_{k+1}$ is integral, and thus if $\varphi_{k+1}(\rho) =1$ implies that $\rho$ is a primitive lattice point. 

In (2), assume that there is some $r$ so that $\rho/r$ is still in $M$. Recall that $\rho = p + \sum_{i=1}^kn_iu_i$ for some $u_i \in E_i$. Thus $\varphi_i(\rho/r) = n_i/r$ is an integer, since $\varphi_i$ is integral and $\rho/r$ is in $M$. Hence $\rho/r - \sum_{i=1}^k (n_i/r)u_i = p/r$ is in $M$. Since $p \in \Sigma[1]$ are assumed to be primitive, $r=1$.
\end{proof}

Finally, this shows that
\begin{corollary}
\label{cor:DeltaV}
Under either of the conditions of Proposition \ref{prop:primitive}, the polytope $\Delta'$ in $M_V$ obtained as the convex hull of the rays generating $ \Sigma_V = M_V \cap \Sigma$ is equal to $\Delta_V$.
\end{corollary}
\begin{proof}
By Proposition \ref{prop:phi1}, each generating ray of $\Delta'$ lies inside of $\Delta_V$, thus the convex hull of the generating rays of $\Sigma_V$ is contained inside of $\Delta_V$. If $\rho$ is a vertex of $\Delta_V$, then let $C$ in $\Sigma$ be the unique cone containing $\rho$ on its interior, $C^0$. If $C_V = C \cap M_V$ then $\rho$ is in $C_V^0$, the interior of $C_V$. Since $\varphi_{k+1}$ is linear on $C_V^0$, we must have some substratum of $\Delta_V$ containing $\rho$ on which $\varphi_{k+1}$ is a linear function, but since $\rho$ is a vertex of $\Delta_V$, the only such substratum is spanned by $\rho$ itself. Thus $C \cap M_V$ is the ray generated by $\rho$ and $\rho$ is in $\Sigma_V[1]$. Therefore all vertices $\rho$ of $\Delta_V$ are in $\Sigma_V[1]$, and hence are primitive, so we can conclude that $\Delta' \subseteq \Delta_V$. 
\end{proof}

It is well known (see e.g. Remark 1.3 of \cite{bn}) that if all facets of an integral polytope $\Delta$ are of integral height $1$ from the origin, then $\Delta$ is reflexive. Thus:

\begin{theorem}
\label{cor:Degen}
Let $X$ be a quasi-Fano complete intersection in a toric variety $Y_\Sigma$, and let $E_1,\dots, E_{k+1}$ be a $\mathbb{Q}$-nef partition of $\Sigma[1]$ so that $E_{k+1}$ corresponds to a nef Cartier divisor. If $X \rightsquigarrow X_{\Sigma_V}$ is defined by an amenable collection of vectors $V$ subordinate to this nef partition, then $X_{\Sigma_V}$ is a weak Fano partial crepant resolution of a Gorenstein Fano toric variety $X_{\Delta_V}$.
\end{theorem}
\begin{proof}
The polytope over the ray generators of $\Sigma_V$ is $\Delta_V$ by Corollary \ref{cor:DeltaV}, which is reflexive by Remark 1.3 of \cite{bn}. It follows that the fan $\Sigma_V$ is a refinement of the fan over faces of $\Delta_V$ obtained without adding rays which are not generated by points in $\Delta_V$. By Lemma 11.4.10 of \cite{cls}, it follows that $X_{\Sigma_V}$ is a crepant partial resolution of $X_{\Delta_V}$ whose anticanonical model is $X_{\Delta_V}$, thus is weak Fano.
\end{proof}

\subsection{Laurent polynomials}
\label{sect:laurent}
We now construct the Givental Landau-Ginzburg model of $X$ as follows \footnote{Our construction of the Landau-Ginzburg model is superficially different from Givental's construction, but the two constructions agree whenever Givental's procedure can be carried out.}. Let us take the usual Laurent polynomial ring in $n$ variables, $\mathbb{C}[x_1^{\pm},\dots,x_n^{\pm}]$ and let $x^\rho$ be elements of the Laurent polynomial ring associated to $\rho \in M$. In particular, if $u_1,\dots,u_n$ is a basis for $N$, then we write
$$
x^\rho = \prod_{i=1}^n x_i^{\langle u_i, \rho \rangle}.
$$
The Givental Landau-Ginzburg model associated to a nef partition $E_1,\dots, E_{k+1}$ of $\Sigma[1]$ is given by the complete intersection $X^\vee$ in $(\mathbb{C}^\times)^n$ written as
\begin{equation}
\label{eq:super}
\sum_{\rho \in E_i} a_\rho x^\rho  = 1
\end{equation}
for $1 \leq i \leq k$ equipped with the superpotential
$$
w = \sum_{\rho \in E_{k+1}} a_\rho x^\rho.
$$
Here $a_\rho$ are constants in $\mathbb{C}^\times$. To be completely correct, the constants $a_\rho$ should be chosen to correspond to complexified classes in the nef cone of $X_\Sigma$. In other words, there should be some integral $\Sigma$-piecewise linear function $\varphi$ and a complex constant $t$ so that 
$$
a_\rho = t^{\varphi(\rho)}.
$$
(see \cite{asgrmo} or \cite{batfan} for details).

The goal of this section is to show that the existence of an amenable collection of vectors subordinate to the $\mathbb{Q}$-nef partition $E_1,\dots, E_{k+1}$ implies that $X^\vee$ is birational to $(\mathbb{C}^\times)^{n-k}$ and that under this birational map, the superpotential $w$ pulls back to a Laurent polynomial. In Section \ref{sect:polytopes} we will determine the relationship between the Laurent polynomial $w$ and the polytope $\Delta_V$.

Let $v_1,\dots, v_{k}$ be an amenable collection of vectors. Then by Proposition \ref{prop:basis}, we may extend $v_1,\dots, v_{k}$ to a basis $v_1,\dots v_n$ of $N$. Let us fix such a basis once and for all. 

Now we can rewrite Equation \ref{eq:super} in terms of this basis as
$$
\sum_{\rho \in E_i} a_\rho x_\rho  = \sum_{\rho \in E_i} a_\rho \left(\prod_{j \geq i} x_{i}^{\langle v_i, \rho \rangle }\right) = 1. 
$$
since $\langle v_i,\rho \rangle = 0$ for $ j < i$. Note that each monomial in this expression can be written as
$$
\prod_{j\geq i} x_{i}^{\langle v_i, \rho \rangle }  = \left( \dfrac{1}{x_i} \right) \prod_{j> i}  x_{i}^{\langle v_i, \rho \rangle },
$$
with non-negative exponents on each $x_j$ for $k \geq j > i$. Thus $x_{k}$ is expressed as a Laurent polynomial in terms of $x_{k+1},\dots, x_n$ and $x_{k-1}$ can be expressed as a Laurent polynomial in terms of $x_{k},\dots, x_n$ with $x_k$ appearing only to non-negative degrees. Substituting in the expression for $x_{k}$ into the resulting expression for $x_{k-1}$, we may express $x_{k-1}$ as a Laurent polynomial in terms of $x_{k+1},\dots, x_{n}$. Repeating this process gives us Laurent polynomial expressions for each $x_{1},\dots, x_{k}$ in terms of $x_{k+1}, \dots, x_{n}$, which we call $f_i(x_{k+1},\dots,x_n)$.

Now we have that, expressed as a function on $(\mathbb{C}^\times)^n$,  
$$
w = \sum_{\rho \in E_{k+1}} \left(\prod_{i=1}^n x_i^{\langle v_i,\rho \rangle}\right)
$$
has $x_1,\dots, x_{k}$ appearing only to positive degrees since $v_1,\dots, v_{k}$ satisfy $\langle v_i, u \rangle \geq 0$ for each $u \in E_{k+1}$. Thus making substitutions $x_i = f_i(x_{k+1},\dots,x_n)$ for each $1 \leq i \leq k$ into $w$, we obtain a Laurent polynomial for $w$ on the variables $x_{k+1},\dots, x_{n}$. We summarize these computations as a theorem.
\begin{theorem}
\label{thm:almost}
Assume $X$ is a quasi-Fano complete intersection in a toric variety $Y_\Sigma$. Then for each amenable toric degeneration $X \rightsquigarrow X_{\Sigma'}$ there is a birational map 
$$
\phi_V: (\mathbb{C}^\times)^{n-k} \dashrightarrow X^\vee
$$ 
so that $\phi_V^*w$ is a Laurent polynomial.
\end{theorem}
\begin{proof}
Let $f_i(x_{k+1},\dots,x_n)$ be the expressions for $x_i$ obtained by using the algorithm described above. We define the map $\phi_V$ as
$$
\phi_V(x_{k+1},\dots,x_n) = (f_1(x_{k+1},\dots,x_n),\dots, f_{k}(x_{k+1},\dots, x_n), x_{k+1},\dots,x_n).
$$
Of course, as a map from $(\mathbb{C}^\times)^{n-k}$ to $(\mathbb{C}^\times)^n$, this map is undefined when $f_i(x_{k+1},\dots,x_n) =0$ for $1\leq i \leq k$, which is a Zariski closed subset of $(\mathbb{C}^\times)^{n-k}$. We have shown above that $\phi_V$ has image which lies inside of $X^\vee$, thus since $\dim(X^\vee) = n-k$, the map $\phi_V$ is a birational map from $(\mathbb{C}^\times)^{n-k}$ to $X^\vee$.
\end{proof}

Thus the choice of an amenable set of vectors $v_1,\dots, v_{k}$ determines both a toric degeneration of $X$ and a Laurent polynomial expression for its Landau-Ginzburg model. In Section \ref{sect:polytopes} we will examine the relationship between the Laurent polynomial $\phi^*_Vw$ and the polytope $\Delta_V$.

\subsection{Comparing polytopes}
\label{sect:polytopes}

Now we will show that if $E_1,\dots, E_{k+1}$ is a $(k+1)$-partite $\mathbb{Q}$-nef partition of a fan $\Sigma$ and $V$ is an amenable collection subordinate to this $\mathbb{Q}$-nef partition, then the Newton polytope of $\varphi^*_Vw$ is precisely $\Delta_V$. Let $\Delta_{\phi^*_Vw}$ be the Newton polytope of the Laurent polynomial $\phi^*_Vw$.

Take any subset $S \subseteq M$, then if we choose $v \in N$, we get stratification of $S$ by the values of $\langle v, - \rangle$. We will define subsets of $S$
$$
S^{b}_{v} = \{ s\in S | \langle v,s \rangle = b \}.
$$
If $S$ is contained in a compact subset of $M \otimes_\mathbb{Z} \mathbb{R}$, then let $b^v_{\mathrm{max}}$ be the maximal value so that $S^b_v$ is nonempty.
\begin{proposition}
\label{prop:sums}
Let $E_1, \dots, E_{k+1}$ be a $\mathbb{Q}$-nef partition of a fan $\Sigma$, and let $V = \{v_1,\dots,v_k \}$ be an amenable collection of vectors in $N$ subordinate to $E_1,\dots, E_{k+1}$. Then the monomials of $\varphi^*_Vw$ correspond to all possible sums of points $p \in E_{k+1}$ and $u_1,\dots, u_\ell$ in $\bigcup_{i=1}^k E_i$ (allowing for repetition in the set $u_1,\dots,u_\ell$) so that 
$$
\langle v_i, p + \sum_{i=1}^\ell u_i \rangle = 0.
$$
for all $1 \leq i \leq k$.
\end{proposition}
\begin{proof}
Let $i=0$. It is clear then that the points in $M$ corresponding to monomials in the Laurent polynomial
$$
w = \sum_{\rho \in E_{k+1}} a_\rho x^\rho
$$
are all points $q$ in $M$ obtained as sums of points $p \in E_{k+1}$ and $u_1,\dots, u_m \in \cup_{j=1}^i E_j$ so that 
$$
\langle v_j, p + \sum_{i=1}^\ell u_i \rangle = 0.
$$
for all $1 \leq j \leq i$. Now we may apply induction.

Let us define
$$
h_i(x_{i+1},\dots, x_n) = \sum_{\rho \in E_i} a_\rho \prod_{j > i} x_j^{\langle v_j, \rho \rangle}
$$
Assume that we have sequentially substituted $h_1,\dots, h_{i-1}$ into $w$ to get a Laurent polynomial $w'$ in the variables $x_i,\dots, x_n$, and that the resulting expression has monomials which correspond to all points $q$ in $M$ which are all sums of points $p$ in $E_{k+1}$ and $u_1,\dots, u_\ell$ in $E_1,\dots, E_{i -1}$ so that $\langle v_j, p + \sum_{s=1}^\ell u_s \rangle =0$ for $1 \leq j \leq i-1$ (allowing for repetition in $u_1,\dots,u_\ell$). Now we show that substituting the expression $h_i(x_{i+1},\dots,x_n)$ into $w'$ gives a polynomial whose monomials correspond to all points $p + \sum_{j=1}^r u_j$ so that $p \in E_{k+1}$, $u_j \in \bigcup_{s=1}^{i+1} E_s$ and $\langle v_{j}, p + \sum_{s=1}^r u_s \rangle =0$ for all $1 \leq j \leq i$ (again, allowing for repetition in the set $u_1,\dots,u_r$). 

We may let $F$ be the set of integral points in $M$ corresponding to monomials of $w'$. Then we have 
$$
w' = \sum_{b=0}^{b_{v_i}^{\mathrm{max}}} x_i^b g_b(x_{i+1},\dots,x_n)
$$
where 
$$
g_b(x_{i+1},\dots,x_n) = \sum_{\rho \in F_{v_i}^{b}} a_\rho \prod_{j>i} x_j^{\langle v_j, \rho \rangle}.
$$
Substituting into $w'$ the expression $x_i = h_i(x_{i+1},\dots,x_n)$ gives us a Laurent polynomial in $x_{i+1},\dots,x_n$ whose monomials correspond to points in set
$$
\cup_{b=0}^{b_{\mathrm{max}}^{v_i}}(F_{v_i}^b + bE_i).
$$ 
Each point in this set satisfies $\langle v_i, E_{v_i}^b + bE_i \rangle=0$ by the condition that $\langle v_i, E_i \rangle =-1$. Furthermore, $\langle v_j, E_i \rangle = 0$ for $j < i$, and hence each set of points $F_{v_i}^b + bE_i$ is orthogonal to $v_1,\dots, v_i$ and can be expressed as a sum of points $p + \sum_{j=1}^s u_i$ for $u_1,\dots, u_s \in \cup_{j=1}^i E_j$ and $p \in E_{k+1}$.

Now assume that we have a point $q = p + \sum_{j=1}^s u_i$ for $u_1,\dots, u_s \in \cup_{j=1}^i E_j$ and $p \in E_{k+1}$ which is orthogonal to $v_j$ for $1 \leq j \leq i$. Then let $U = \{u_1,\dots, u_s \} \cap E_i$, and let 
$$
q' = p + \sum_{i=1, u_i \notin U}^s u_i = \rho - \sum_{u_i \in U} u_i.
$$
We see that $q'$ is orthogonal to $v_j$ for $1 \leq j \leq i-1$ since $\langle v_j, u \rangle =0$ for $u\in U$ and $1\leq j \leq i-1$, thus $q' \in F$. Note that we must have $\langle v_i, q' \rangle = \# U$. Thus the point $q'$ is in $F_{v_i}^{\#U}$ and hence $q \in (F_{v_i}^{\#U} + (\#U)E_i)$ (since $\sum_{u_i \in U}u_i$ is clearly an element of $(\# U)$). Thus $q$ corresponds to a monomial in $w'$ after making the substitution $x_i = h_i(x_{i+1},\dots,x_n)$. This completes the proof after applying induction.
\end{proof}

\begin{proposition}
\label{prop:Vinf}
Assume that $M_V$ intersects a cone $C$ of $\Sigma$ in a ray generated by an integral vector $\rho \in M$, then there is some multiple of $\rho$ in the polytope $\Delta_{\phi^*_Vw}$. In other words, $\Delta_V \subseteq \Delta_{\phi^*_Vw}$
\end{proposition}
\begin{proof}
This follows from Proposition \ref{prop:sums} and Proposition \ref{prop:phi1}. According to Proposition \ref{prop:phi1}, there is an integral point in $M_V \cap C$ so that $\varphi_{k+1}(\rho) = 1$. In the proof of Proposition \ref{prop:phi1}, it is actually shown that this point is constructed as a sum of points $p \in E_{k+1}$ and $u_1,\dots, u_\ell \in \bigcup_{i=1}^k E_i$ (allowing for repetition in the set $u_1,\dots, u_\ell$). According to Proposition \ref{prop:sums} this point must correspond to a monomial of the Laurent polynomial $\phi_V^*w$.
\end{proof}

\begin{theorem}
\label{thm:f=v}
Assume that $V$ is an amenable collection of vectors subordinate to a $\mathbb{Q}$-nef partition $E_1,\dots, E_{k+1}$ of a fan $\Sigma$. The polytope $\Delta_{\phi_V^*w}$ is equal to $\Delta_V$.
\end{theorem}
\begin{proof}
We may deduce that $\Delta_V \subseteq \Delta_{\phi^*_Vw}$ from Proposition \ref{prop:Vinf}. Thus it is sufficient to show that $\Delta_{\phi^*_Vw}$ is contained in $\Delta_V$, or in other words, each integral point $\rho \in \Delta_{\phi^*_Vw}$ satisfies $\varphi_{k+1}(p) \leq 1$. However, this is reasonably easy to see. We have shown in Proposition \ref{prop:sums} that each point in $\Delta_{\phi_V^*w}$ is a sum of a single point $p \in E_{k+1}$ and a set of points $u_1,\dots, u_\ell$ in $\Sigma[1] \setminus E_{k+1}$ (allowing for repetition in the set $u_1,\dots,u_\ell$). Recall that we have a set of vectors $w_1,\dots, w_v \in N$ for $v$ the number of maximal dimensional faces of $\Sigma_\Delta$ so that 
$$
\varphi_{k+1}(\rho) = \mathrm{max} \{\langle w_i, \rho \rangle \}_{i=1}^v. 
$$
Now let us apply this to $\rho = p + \sum_{i=1}^\ell u_i$. We obtain
$$
\mathrm{max} \{\langle w_i, p + \sum_{i=1}^\ell u_i \rangle \}_{i=1}^v \leq   \mathrm{max} \{ \langle w_i,  p \rangle \}_{i=1}^v + \sum_{i=1}^\ell \left( \mathrm{max} \{ \langle w_i,  u_i \rangle \}_{i=1}^v \right) = \varphi_{k+1}(p) = 1
$$
as required.
\end{proof}
Note that this is actually a general description of the polytope $\Delta_{\phi^*w}$ without any restrictions on the $\mathbb{Q}$-nef partition. We summarize the results of this section as the following theorem, which follows directly from Theorem \ref{cor:DeltaV} and Theorem \ref{thm:f=v}.

\begin{theorem}
\label{thm:main}
Let $X$ be a complete intersection in a toric variety $Y_\Sigma$ so that there is a $\mathbb{Q}$-nef partition $E_1,\dots, E_{k+1}$ of $\Sigma[1]$ so that $X$ is a complete intersection of $\mathbb{Q}$-Cartier divisors determined by $E_1,\dots,E_k$, then if
\begin{enumerate}
\item $E_{k+1}$ is a Cartier divisor or
\item $E_1,\dots, E_k$ are Cartier divisors,
\end{enumerate}
then an amenable collection of vectors $V$ subordinate to this $\mathbb{Q}$-nef partition determines an amenable degeneration $X \rightsquigarrow X_{\Sigma_V}$ for some fan $\Sigma_V$, and the corresponding Laurent polynomial has Newton polytope equal to the convex hull of $\Sigma_V[1]$.
\end{theorem}

A more robust geometric statement is available to us in the case where $X$ is a Fano variety corresponding to a nef partition in a toric variety. This follows from Theorem \ref{thm:main} and Theorem \ref{cor:Degen}

\begin{theorem}
\label{thm:reflexfano}
Assume $X$ is a Fano toric complete intersction in a toric variety $Y = Y_\Delta$ cut out by the vanishing locus of sections $s_i \in H^0(\mathcal{O}_Y(E_i),Y)$ for $1\leq i \leq k$, and $E_1,\dots, E_{k+1}$ is a nef partition of $\Delta$, and that $V$ is an amenable collection of vectors subordinate to this nef partition. Then $V$ determines:
\begin{enumerate}
\item A degeneration of $X$ to a toric variety $\widetilde{X}_{\Delta_V}$ which is a crepant partial resolution of of $X_{\Delta_V}$ and
\item A birational map $\phi_V: (\mathbb{C}^\times)^{n-k} \dashrightarrow X^\vee$ so that $\phi_V^*w$ has Newton polytope equal to $\Delta_V$.
\end{enumerate}
\end{theorem}

\subsection{Mutations}
Here we will analyze the relationship between Laurent polynomials obtained from the same nef partition and different amenable collections. First we recall the following definition from \cite{gu}.

\begin{defn}\label{def:mut}
Let $f$ be a Laurent polynomial in $n$ variables and let 
$$
\omega_n = \dfrac{dx_1 \wedge \dots \wedge d x_n }{(2\pi i )^n x_1\dots x_n}
$$ 
A {\it mutation} of $f$ is a birational map $\phi: (\mathbb{C}^\times)^n \dashrightarrow (\mathbb{C}^\times)^n$ so that $\phi^*\omega = \omega$ and so that $\phi^*f$ is again a Laurent polynomial.
\end{defn}

Assume, first of all, that we have two different amenable collections $V$ and $V'$ which are subordinate to the same nef partition $E_1,\dots,E_{k+1}$ of a fan $\Sigma$. Then associated to both $V$ and $V'$ are two maps. The first map is 
$$
\phi_V : (\mathbb{C}^\times)^{n-k} \dashrightarrow X^\vee
$$
and the second is a map
$$
\phi^{-1}_V = \pi_V : (\mathbb{C}^\times)^{k} \rightarrow (\mathbb{C}^\times)^{n-k}
$$
which is defined as
$$
(x_1,\dots, x_n) \mapsto (x_{k+1}, \dots, x_n).
$$
so that $\phi_V$ is a birational section of $\pi_V$. However, the map $\pi_V\cdot \phi_{V'}$ for a different amenable collection $V'$ is simply a birational morphism of tori. If we let $y_{k+1},\dots, y_n$ and $x_{k+1},\dots, x_n$ be coordinates on the torus $(\mathbb{C}^\times)^{n-k}$ associated to $V$ and $V'$ respectively, then for each $k+1 \leq j \leq n$, there is a rational polynomial $h_j(x_{k+1},\dots, x_n)$ so that the map $\pi_V \cdot \phi_{V'}$ is written as
$$
(x_{k+1},\dots ,x_n) \mapsto (h_{k+1},\dots, h_n).
$$
In particular, to determine this map, we have Laurent polynomials $f_i(x_{k+1},\dots,x_n)$ for each $1 \leq i \leq k$ so that 
$$
\phi_{V'}(x_{k+1},\dots,x_n) = (f_1(x_{k+1},\dots, x_n),\dots, f_k(x_{k+1},\dots, x_n), x_{k+1},\dots, x_n).
$$
There are bases $B$ and $B'$ of $N$ associated to both $V$ and $V'$ so that $B = \{v_1,\dots, v_n\}$ and $V = \{ v_1,\dots, v_k\}$ and so that $B = \{u_1,\dots, u_n\}$ with $V' = \{u_1,\dots,u_n\}$. There is an invertible matrix $Q$ with integral entries $q_{i,j}$ so that $v_i = \sum_{j =1}^n q_{i,j}u_j$, and torus coordinates $x_1,\dots, x_n$ and $y_1,\dots,y_n$ on $(\mathbb{C}^\times)^n$ related by 
$$
y_i = \prod_{j=1}^n x_j^{q_{i,j}}.
$$
In particular, we have 
$$
h_i (x_{n-k+1},\dots,x_n) = \left(\prod_{j=1}^{k} f_j(x_{k+1},\dots,x_n)^{q_{i,j}}\right) \left( \prod_{j=k+1}^n x_j^{q_{i,j}}\right).
$$
The map given by the polynomials $h_i(x_{n-k+1},\dots,x_n)$ for $k+1 \leq i \leq n$ then determine explicitly the birational morphism above associated to a pair of amenable collections subordinate to a fixed nef partition. Now it is clear that the birational map of tori $\phi_{V'}^{-1} \cdot \phi_{V}$ induces a birational map of tori which pulls back the Laurent polynomial $\phi_V^* w$ to a Laurent polynomial. Our goal now is to show that this map preserves the torus invariant holomorphic $n$ form $\omega_{n-k}$ defined in Definition \ref{def:mut}. First, we prove a lemma.

\begin{lemma}\label{lem:mut}
Let $\phi:(\mathbb{C}^\times)^n \dashrightarrow X^\vee \subseteq (\mathbb{C}^\times)^n$ be a birational map onto a complete intersection in $(\mathbb{C}^\times)^n$ cut out by Laurent polynomials 
$$
F_i = 1-\left(\frac{1}{x_i}\right)f_i(x_{i+1},\dots,x_n)
$$ 
which have only non-negative exponents in $x_{i+1},\dots, x_k$ if $i \neq k$. Then the residue
$$
\mathrm{Res}_{X^\vee} \left(\dfrac{\omega_n}{F_1 \dots F_k}\right)
$$
agrees with the form $(2\pi i)^k \omega_{n-k}$ on the domain of definition of $\phi$.
\end{lemma}
\begin{proof}
We argue by induction. We may make a birational change of variables on $(\mathbb{C}^\times)^n$ so that $x_1 = y_1 + f_1(y_{k+1},\dots,y_n)$ for each $1 \leq i \leq k$ and $y_i = x_i$ for $i \neq 1$. Note that 
$$
dx_1 = d\left(y_1 + f(y_{2},\dots, y_n)\right) = dy_1 + \rho
$$
where $\rho$ is some $1$-form written as a linear combination of $dy_{2},\dots, dy_n$ with Laurent polynomial coefficients. Thus under our change of variables,
\begin{align*}
dx_1   \wedge \dots \wedge dx_n &=  (dy_1 + \rho) \wedge \dots \wedge dy_n \\ 
&= dy_1 \wedge \dots \wedge dy_n.
\end{align*}
Furthermore, under the correct choice of variables, we have
$$
F_1(x_1,\dots,x_n) = 1- \left(\dfrac{1}{x_1}\right) f_1(x_{2},\dots, x_n) 
$$
Thus 
$$
\dfrac{\omega_n}{F_1\dots, F_k} = \dfrac{dx_1 \wedge \dots \wedge dx_n}{(x_1-f_1(x_2,\dots,x_k))F_2\dots F_{k}x_2 \dots x_n}.
$$
Changing variables to $y_1,\dots, y_n$, we see that
$$
\dfrac{\omega_n}{F_1 \dots F_k} = \dfrac{dy_1 \wedge \dots \wedge dy_n}{F_2\dots F_k y_1 y_2 \dots, y_n}
$$
whose residue along the locus $y_1 = 0$ (which is precisely the image of our torus embedding $\phi$), is just $\frac{(2\pi i)\omega_{n-1}}{F_2\dots F_k}$ since $F_2,\dots,F_k$ are independent of $y_1$. Thus locally around any point in $X^\vee$ where the birational map $\phi$ is well defined and the torus change of coordinates $\varphi$ is well-defined, it follows that the residue of $\frac{\omega_n}{F_1\dots F_2}$ on $X^\vee$ agrees with $\frac{(2\pi i)\omega_{n-1}}{F_2\dots F_k}$. Repeating this argument for each $2 \leq i \leq k$ shows that
$$
\phi_V^*\mathrm{Res}_{X^\vee} \left(\dfrac{\omega_n}{F_1 \dots F_k}\right) = (2\pi i)^k \omega_{n-k}.
$$
\end{proof}

Now this allows us to prove:

\begin{theorem}\label{thm:mut}
Let $V$ and $V'$ be two amenable collections of vectors subordinate to a nef partition $E_1,\dots, E_{k+1}$. Then the birational map of tori $\phi_V^{-1} \cdot \phi_{V'}$ is a mutation of the Laurent polynomial $\phi_{V'}^*w$.
\end{theorem}
\begin{proof}
It is clear that this map is birational and takes $\phi_{V'}^*w$ to a Laurent polynomial. To see that $\phi_V^{-1} \cdot \phi_{V'}$ preserves the form $\omega_{n-k}$, we note that there is some open subset $U^\vee$ of $X^\vee$ on which both $\phi_V$ and $\phi_{V'}$ induce isomorphisms from open sets $U_V$ and $U_{V'}$ in $(\mathbb{C}^\times)^{n-k}$. In other words, we have isomorphisms $\phi_V^\circ : U_V \xrightarrow{\sim} U^\vee$ and $\phi_{V'}^\circ: U_{V'} \xrightarrow{\sim} U^\vee$ between open sets. From Lemma \ref{lem:mut}, we know that 
\begin{align*}
(\phi^\circ_V)^* \mathrm{Res}_{U^\vee}\left( \frac{\omega_n}{F_1\dots F_k} \right) & = (2\pi_i)^k \omega_{n-k}|_{U_V}\\
(\phi^\circ_{V'})^*\mathrm{Res}_{U^\vee}\left(\frac{\omega_n}{F_1\dots F_k} \right) &= (2\pi_i)^k \omega_{n-k}|_{U_{V'}}
\end{align*}
therefore we must have that on $U_V$, $(\phi_V \cdot \phi_{V'}^{-1})^*(\omega_{n-k}|_{U_{V'}}) = \omega_{n-k}|_{V}$, and thus $\phi_V \cdot \phi_V^{-1}$ is a mutation.
\end{proof}

Of course, Theorem \ref{thm:mut} requires that we start with two amenable collections subordinate to the same nef partition. It is possible to have distinct nef partitions corresponding to the same quasi-Fano variety. It would be interesting to show that if we have two such nef partitions and amenable collections subordinate to each, then there is a mutation between the corresponding Laurent polynomials.

\section{Degenerations of complete intersections in partial flag varieties}
\label{sect:flag}
Now we discuss the question of constructing toric degenerations and Laurent polynomial expressions for Landau-Ginzburg models of complete intersections in partial flag varieties. Recall that the partial flag variety $F(n_1,\dots, n_l,n)$ is a smooth complete Fano variety which parametrizes flags in $V \cong \mathbb{C}^n$,
$$
V_1 \subseteq \dots \subseteq V_l \subseteq V
$$
where $\dim(V_i) = n_i$. The reader may consult \cite{brion} for general facts on partial flag varieties.

According to \cite{bcfkvs1} and \cite{gs}, there are small toric degenerations of the complete flag variety $F(n,n_1,\dots,n_l)$ to Gorenstein Fano toric varieties $P(n,n_1,\dots,n_l)$ which admit small resolutions of singularities. It is suggested in \cite{ps} that the Landau-Ginzburg models of the complete flag variety can be expressed as a Laurent polynomial whose Newton polytope is the polytope $\Delta(n_1,\dots,n_l,n)$ whose face fan determines the toric variety $P(n,n_1,\dots,n_l)$. 

For any Fano complete intersection $X$ in $F(n_1,\dots,n_l,n)$, one obtains a degeneration of $X$ to a nef Cartier complete intersection in the toric variety $P(n,n_1,\dots,n_l)$ and hence conjectural expressions for the Landau-Ginzburg model of $X$ can be given in terms of the Givental Landau-Ginzburg model of complete intersections in $P(n, n_1,\dots,n_l)$. In \cite{ps}, Przyjalkowski and Shramov give a method of constructing birational maps between tori and $X^\vee$ so that the superpotential pulls back to a Laurent polynomial for complete intersections in Grassmannians $\mathrm{Gr}(2,n)$. Here we will use Theorem \ref{thm:main} to show that most nef complete intersections $X$ in $P(n,n_1,\dots,n_l)$ admit an amenable toric degeneration, which express the Givental Landau-Ginzburg model of $X$ as a Laurent polynomial.

\subsection{The structure of $P(n_1,\dots,n_l,n)$}
\label{sect:Pn1}
In order to construct the toric variety to which $F(n_1,\dots,n_l,n)$ degenerates, we begin with an external combinatorial construction presented in \cite{bcfkvs1}. We define a graph $\Gamma(n_1,\dots,n_l,n)$. Let us take an $n \times n$ box in the Euclidean plane with lower left corner placed at the point $(-1/2,-1/2)$. Let $k_{l +1}= n - n_l$, let $k_i = n_{i} - n_{i-1}$, and $k_1 = n_1$. Along the diagonal of this box moving from the bottom right corner to the top left corner, we place boxes of size $k_i \times k_i$ sequentially from $1$ to ${l+1}$. The region below these boxes is then divided equally into $1 \times 1$ boxes along grid lines, as shown in the first part of Figure \ref{fig:graph}.

\begin{figure}

\begin{center}
\begin{tikzpicture}[scale = 0.6]

\draw[dashed] (-0.5,7.5) to (7.5,7.5);
\draw[dashed] (-0.5,7.5) to (-0.5,-0.5);
\draw[dashed] (7.5,7.5) to (7.5,-0.5);
\draw[dashed] (-0.5,-0.5) to (7.5,-0.5);

\draw[dashed] (2.5,-0.5) to (2.5,7.5);
\draw[dashed] (7.5,7.5) to (7.5,-0.5);

\draw[dashed] (-0.5,4.5) to (5.5,4.5);
\draw[dashed] (-0.5,3.5) to (2.5,3.5);
\draw[dashed] (-0.5,2.5) to (2.5,2.5);
\draw[dashed] (-0.5,1.5) to (7.5,1.5);
\draw[dashed] (-0.5,0.5) to (5.5,0.5);

\draw[dashed] (5.5,4.5) to (5.5,-0.5);

\draw[dashed] (0.5,4.5) to (0.5,-0.5);
\draw[dashed] (1.5,4.5) to (1.5,-0.5);

\draw[dashed] (3.5,1.5) to (3.5,-0.5);
\draw[dashed] (4.5,1.5) to (4.5,-0.5);
\end{tikzpicture}
%\end{center}
%\caption{The grid for $\Gamma(2,5,8,15)$}
%\end{figure}
%
%\begin{figure}
%
%\begin{center}
\begin{tikzpicture}[scale = 0.6]

% Nodes
\draw(0,0) node[circle,fill,inner sep=1pt](){};
\draw(0,1) node[circle,fill,inner sep=1pt](){};
\draw(0,2) node[circle,fill,inner sep=1pt](){};
\draw(0,3) node[circle,fill,inner sep=1pt](){};
\draw(0,4) node[circle,fill,inner sep=1pt](){};

\draw(1,0) node[circle,fill,inner sep=1pt](){};
\draw(1,1) node[circle,fill,inner sep=1pt](){};
\draw(1,2) node[circle,fill,inner sep=1pt](){};
\draw(1,3) node[circle,fill,inner sep=1pt](){};
\draw(1,4) node[circle,fill,inner sep=1pt](){};

\draw(2,0) node[circle,fill,inner sep=1pt](){};
\draw(2,1) node[circle,fill,inner sep=1pt](){};
\draw(2,2) node[circle,fill,inner sep=1pt](){};
\draw(2,3) node[circle,fill,inner sep=1pt](){};
\draw(2,4) node[circle,fill,inner sep=1pt](){};

\draw(3,1) node[circle,fill,inner sep=1pt](){};
\draw(3,0) node[circle,fill,inner sep=1pt](){};

\draw(4,1) node[circle,fill,inner sep=1pt](){};
\draw(4,0) node[circle,fill,inner sep=1pt](){};

\draw(5,1) node[circle,fill,inner sep=1pt](){};
\draw(5,0) node[circle,fill,inner sep=1pt](){};

\draw(6,0) node[circle,inner sep=1pt, draw](){};

\draw(0,5) node[circle,inner sep=1pt, draw](){};
\draw(3,2) node[circle,inner sep=1pt, draw](){};

%Grid
\draw[dashed] (-0.5,7.5) to (7.5,7.5);
\draw[dashed] (-0.5,7.5) to (-0.5,-0.5);
\draw[dashed] (7.5,7.5) to (7.5,-0.5);
\draw[dashed] (-0.5,-0.5) to (7.5,-0.5);

\draw[dashed] (2.5,-0.5) to (2.5,7.5);
\draw[dashed] (7.5,7.5) to (7.5,-0.5);

\draw[dashed] (-0.5,4.5) to (5.5,4.5);
\draw[dashed] (-0.5,3.5) to (2.5,3.5);
\draw[dashed] (-0.5,2.5) to (2.5,2.5);
\draw[dashed] (-0.5,1.5) to (7.5,1.5);
\draw[dashed] (-0.5,0.5) to (5.5,0.5);

\draw[dashed] (5.5,4.5) to (5.5,-0.5);

\draw[dashed] (0.5,4.5) to (0.5,-0.5);
\draw[dashed] (1.5,4.5) to (1.5,-0.5);

\draw[dashed] (3.5,1.5) to (3.5,-0.5);
\draw[dashed] (4.5,1.5) to (4.5,-0.5);
\end{tikzpicture}
\begin{tikzpicture}[scale=0.6]
%Arrows
\draw[->] (0,4.95) to (0,4.05);
\draw[->] (0,4) to (0,3.05);
\draw[->] (0,3) to (0,2.05);
\draw[->] (0,2) to (0,1.05);
\draw[->] (0,1) to (0,0.05);
\draw[->] (0,4) to (0.95,4);
\draw[->] (1,4) to (1.95,4);
\draw[->] (1,4) to (1,3.05);
\draw[->] (1,3) to (1,2.05);
\draw[->] (1,2) to (1,1.05);
\draw[->] (1,1) to (1,0.05);
\draw[->] (2,4) to (2,3.05);
\draw[->] (2,3) to (2,2.05);
\draw[->] (2,2) to (2,1.05);
\draw[->] (2,1) to (2,0.05);
\draw[->] (0,3) to (0.95,3);
\draw[->] (1,3) to (1.95,3);
\draw[->] (0,2) to (.95,2); 
\draw[->] (1,2) to (1.95,2); 
\draw[->] (2,2) to (2.95,2);
\draw[->] (0,1) to (0.95,1); \draw[->] (1,1) to (1.95,1); \draw[->] (2,1) to (2.95,1); \draw[->] (3,1) to (3.95,1);\draw[->] (4,1) to (4.95,1);
\draw[->] (3,1.95) to (3,1.05); \draw[->] (3,1) to (3,0.05);
\draw[->] (0,0) to (.95,0); \draw[->] (1,0) to (1.95,0); \draw[->] (2,0) to (2.95,0); \draw[->] (3,0) to (3.95,0);\draw[->] (4,0) to (4.95,0);\draw[->] (5,0) to (5.95,0);
\draw[->] (4,1) to (4,0.05); \draw[->] (5,1) to (5,0.05);

% Nodes
\draw(0,0) node[circle,fill,inner sep=1pt](){};
\draw(0,1) node[circle,fill,inner sep=1pt](){};
\draw(0,2) node[circle,fill,inner sep=1pt](){};
\draw(0,3) node[circle,fill,inner sep=1pt](){};
\draw(0,4) node[circle,fill,inner sep=1pt](){};

\draw(1,0) node[circle,fill,inner sep=1pt](){};
\draw(1,1) node[circle,fill,inner sep=1pt](){};
\draw(1,2) node[circle,fill,inner sep=1pt](){};
\draw(1,3) node[circle,fill,inner sep=1pt](){};
\draw(1,4) node[circle,fill,inner sep=1pt](){};

\draw(2,0) node[circle,fill,inner sep=1pt](){};
\draw(2,1) node[circle,fill,inner sep=1pt](){};
\draw(2,2) node[circle,fill,inner sep=1pt](){};
\draw(2,3) node[circle,fill,inner sep=1pt](){};
\draw(2,4) node[circle,fill,inner sep=1pt](){};

\draw(3,1) node[circle,fill,inner sep=1pt](){};
\draw(3,0) node[circle,fill,inner sep=1pt](){};

\draw(4,1) node[circle,fill,inner sep=1pt](){};
\draw(4,0) node[circle,fill,inner sep=1pt](){};

\draw(5,1) node[circle,fill,inner sep=1pt](){};
\draw(5,0) node[circle,fill,inner sep=1pt](){};

\draw(6,0) node[circle,inner sep=1pt, draw](){};

\draw(0,5) node[circle,inner sep=1pt, draw](){};
\draw(3,2) node[circle,inner sep=1pt, draw](){};

%Grid
\draw[dashed] (-0.5,7.5) to (7.5,7.5);
\draw[dashed] (-0.5,7.5) to (-0.5,-0.5);
\draw[dashed] (7.5,7.5) to (7.5,-0.5);
\draw[dashed] (-0.5,-0.5) to (7.5,-0.5);

\draw[dashed] (2.5,-0.5) to (2.5,7.5);
\draw[dashed] (7.5,7.5) to (7.5,-0.5);

\draw[dashed] (-0.5,4.5) to (5.5,4.5);
\draw[dashed] (-0.5,3.5) to (2.5,3.5);
\draw[dashed] (-0.5,2.5) to (2.5,2.5);
\draw[dashed] (-0.5,1.5) to (7.5,1.5);
\draw[dashed] (-0.5,0.5) to (5.5,0.5);

\draw[dashed] (5.5,4.5) to (5.5,-0.5);

\draw[dashed] (0.5,4.5) to (0.5,-0.5);
\draw[dashed] (1.5,4.5) to (1.5,-0.5);

\draw[dashed] (3.5,1.5) to (3.5,-0.5);
\draw[dashed] (4.5,1.5) to (4.5,-0.5);
\end{tikzpicture}

\end{center}
\caption{The grid, nodes and graph of $\Gamma(2,5,8)$}
\label{fig:graph}
\end{figure}
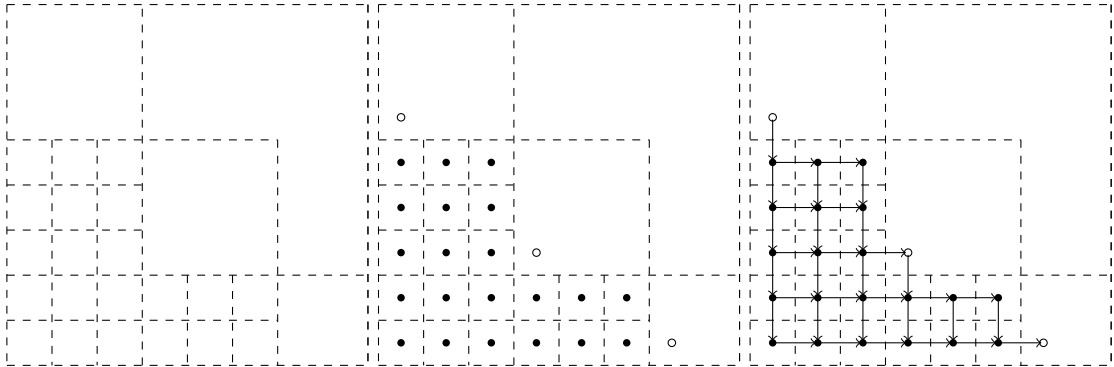
From this grid, we construct a directed graph with black and white vertices. Assume that the centers of each of the $1 \times 1$ boxes beneath the diagonal are at integral points in the $(x,y)$ plane so that the center of the bottom left box is at the origin. At the center of each $1 \times 1$ box beneath the diagonal, we place black points. In each box $B$ on the diagonal, we insert a white point shifted by $(1/2,1/2)$ from the bottom left corner of $B$. See the second part of Figure \ref{fig:graph} as an example.

We then draw arrows between each vertex $u$ and any other vertex $v$ which can be obtained from $u$ by a shift of $v$ by either $(1,0)$ or $(0,-1)$ directed from left to right or from top to bottom, as in the third part of Figure \ref{fig:graph}. Let $D$ be the set of black vertices, and let $S$ be the set of white vertices. In the language of \cite{bcfkvs1}, the elements of $S$ are called stars. Let $E$ denote the set of edges of $\Gamma(n_1,\dots,n_l,n)$. We will denote the vertex at a point $(m,n) \in \mathbb{Z}_{\geq 0}^2$ by $v_{m,n}$ and an arrow between points $v_{m_1,n_1}$ and $v_{m_2,n_2}$ by $(v_{m_1,n_1} \rightarrow v_{m_2,n_2})$. We have functions 
$$
h : E \rightarrow D \cup S \text{ and } t: E \rightarrow D \cup S
$$
assigning to an arrow the vertex corresponding to its head and tail respectively.

The polytope $\Delta(n_1,\dots,n_l,n)$ is then constructed as a polytope in the lattice $M = \mathbb{Z}^{D}$ as the convex hull of points corresponding to edges $E$, which we construct as follows. If $d \in D$, then let $e_d$ be the associated basis vector for $M$, and formally define $e_s$ to be the origin for $s \in S$. If $\alpha$ is an edge of $\Gamma(n_1,\dots,n_l,n)$, then to $\alpha$, we associate the point in $M$ given by 
$$
p_{\alpha} = e_{h(\alpha)} - e_{t(\alpha)}.
$$

\begin{defn}
The polytope $\Delta(n_1,\dots,n_l,n)$ is the convex hull of the points $p_\alpha$ for all $\alpha \in E$.
\end{defn}

We rapidly review properties of $\Delta(n_1,\dots,n_l,n)$. The toric variety $P(n_1,\dots,n_l,n)$ is toric variety associated to the fan over faces of $\Delta(n_1,\dots,n_l,n)$, and it has torus invariant Weil divisors associated to each vertex $v$, which correspond directly to the points $p_\alpha$ for $\alpha \in E$. We will refer to the divisor corresponding to the arrow $\alpha$ as $D_\alpha$.

Torus invariant Cartier divisors $\sum_{\alpha \in \Delta[0]} n_\alpha D_\alpha$ correspond to piecewise linear functions $\varphi$ which are $\Sigma$-linear so that $\varphi(q_\alpha) = n_\alpha$ for all $q_\alpha$. In Lemma 3.2.2 of \cite{bcfkvs1}, Cartier divisors which generate $\Pic(P(n_1,\dots,n_l,n))$ are given. We will now describe these divisors.

\begin{defn}
A {\it roof} $\mathcal{R}_i$ for $i \in \{1,\dots, l \}$ is a collection of edges which have either no edges above or to the right, and which span a path between two sequential white vertices of $\Gamma(n_1,\dots,n_l,n)$.
\end{defn}
\begin{figure}
\begin{center}
\begin{tikzpicture}[scale=0.8]

%Arrows
\draw[ultra thick,->] (0,4.95) to (0,4.05);

\draw[ultra thick,->] (0,4) to (0.95,4);
\draw[ultra thick, ->] (1,4) to (1.95,4);

\draw[ultra thick, ->] (2,4) to (2,3.05);
\draw[ultra thick, ->] (2,3) to (2,2.05);
 
\draw[ultra thick, ->] (2,2) to (2.95,2);
% \draw[->] (3,1) to (3.95,1);\draw[->] (4,1) to (4.95,1);
%\draw[->] (3,1.95) to (3,1.05);
%
%\draw[->] (5,0) to (5.95,0);
%
%\draw[->] (5,1) to (5,0.05);

% Nodes
\draw(0,0) node[circle,fill,inner sep=1pt](){};
\draw(0,1) node[circle,fill,inner sep=1pt](){};
\draw(0,2) node[circle,fill,inner sep=1pt](){};
\draw(0,3) node[circle,fill,inner sep=1pt](){};
\draw(0,4) node[circle,fill,inner sep=1pt](){};

\draw(1,0) node[circle,fill,inner sep=1pt](){};
\draw(1,1) node[circle,fill,inner sep=1pt](){};
\draw(1,2) node[circle,fill,inner sep=1pt](){};
\draw(1,3) node[circle,fill,inner sep=1pt](){};
\draw(1,4) node[circle,fill,inner sep=1pt](){};

\draw(2,0) node[circle,fill,inner sep=1pt](){};
\draw(2,1) node[circle,fill,inner sep=1pt](){};
\draw(2,2) node[circle,fill,inner sep=1pt](){};
\draw(2,3) node[circle,fill,inner sep=1pt](){};
\draw(2,4) node[circle,fill,inner sep=1pt](){};

\draw(3,1) node[circle,fill,inner sep=1pt](){};
\draw(3,0) node[circle,fill,inner sep=1pt](){};

\draw(4,1) node[circle,fill,inner sep=1pt](){};
\draw(4,0) node[circle,fill,inner sep=1pt](){};

\draw(5,1) node[circle,fill,inner sep=1pt](){};
\draw(5,0) node[circle,fill,inner sep=1pt](){};

\draw(6,0) node[circle,inner sep=1pt, draw](){};

\draw(0,5) node[circle,inner sep=1pt, draw](){};
\draw(3,2) node[circle,inner sep=1pt, draw](){};

\end{tikzpicture}
\begin{tikzpicture}[scale=0.8]

%Arrows
%\draw[->] (0,4.95) to (0,4.05);
%
%\draw[->] (0,4) to (0.95,4);
%\draw[->] (1,4) to (1.95,4);

%\draw[->] (2,4) to (2,3.05);
%\draw[->] (2,3) to (2,2.05);
 
%\draw[->] (2,2) to (2.95,2);
 \draw[ultra thick, ->] (3,1) to (3.95,1);
 \draw[ultra thick, ->] (4,1) to (4.95,1);
\draw[ultra thick, ->] (3,1.95) to (3,1.05);

\draw[ultra thick, ->] (5,0) to (5.95,0);

\draw[ultra thick, ->] (5,1) to (5,0.05);

% Nodes
\draw(0,0) node[circle,fill,inner sep=1pt](){};
\draw(0,1) node[circle,fill,inner sep=1pt](){};
\draw(0,2) node[circle,fill,inner sep=1pt](){};
\draw(0,3) node[circle,fill,inner sep=1pt](){};
\draw(0,4) node[circle,fill,inner sep=1pt](){};

\draw(1,0) node[circle,fill,inner sep=1pt](){};
\draw(1,1) node[circle,fill,inner sep=1pt](){};
\draw(1,2) node[circle,fill,inner sep=1pt](){};
\draw(1,3) node[circle,fill,inner sep=1pt](){};
\draw(1,4) node[circle,fill,inner sep=1pt](){};

\draw(2,0) node[circle,fill,inner sep=1pt](){};
\draw(2,1) node[circle,fill,inner sep=1pt](){};
\draw(2,2) node[circle,fill,inner sep=1pt](){};
\draw(2,3) node[circle,fill,inner sep=1pt](){};
\draw(2,4) node[circle,fill,inner sep=1pt](){};

\draw(3,1) node[circle,fill,inner sep=1pt](){};
\draw(3,0) node[circle,fill,inner sep=1pt](){};

\draw(4,1) node[circle,fill,inner sep=1pt](){};
\draw(4,0) node[circle,fill,inner sep=1pt](){};

\draw(5,1) node[circle,fill,inner sep=1pt](){};
\draw(5,0) node[circle,fill,inner sep=1pt](){};

\draw(6,0) node[circle,inner sep=1pt, draw](){};

\draw(0,5) node[circle,inner sep=1pt, draw](){};
\draw(3,2) node[circle,inner sep=1pt, draw](){};

\end{tikzpicture}
\end{center}
\caption{Roof paths of $\Gamma(2,5,8)$ connecting sequential white vertices.}
\label{fig:roofs}
\end{figure}
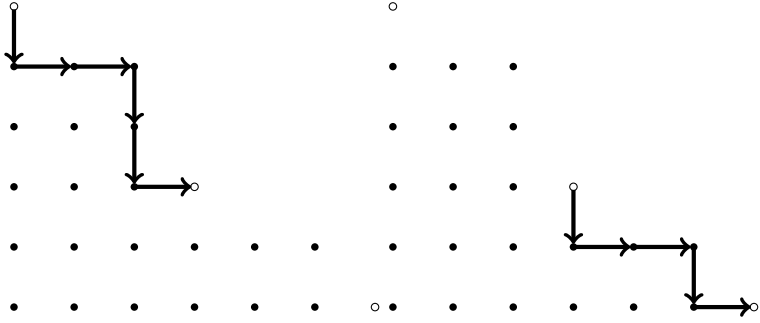
Examples of roofs and the associated paths in $\Gamma(n_1,\dots,n_l,n)$ are shown in Figure \ref{fig:roofs}. Associated to each roof is a set of divisors. Let $\alpha$ be an edge in $\mathcal{R}_i$ and let $U(\alpha)$ be the collection of edges either directly below $\alpha$ if $\alpha$ is a horizontal arrow, or directly to the left of $\alpha$ if $\alpha$ is a vertical arrow. If $D_\beta$ is the Weil divisor of $P(n_1,\dots,n_l,n)$ corresponding to the arrow $\beta$ then it is proven in Lemma 3.2.2 of \cite{bcfkvs1} that the Weil divisor
$$
H_\alpha = \sum_{\beta \in U(\alpha)} D_\beta
$$
is nef and Cartier, and that if we take two edges $\alpha$ and $\alpha'$ in the same roof $\mathcal{R}_i$, then $H_\alpha$ is linearly equivalent to $H_{\alpha'}$. We define $\mathcal{L}_i$ to be the line bundle on $P(n_1,\dots,n_l,n)$ associated to the divisor $H_\alpha$ for $\alpha$ any arrow in $\mathcal{R}_i$. There is an embedding 
$$
\psi: P(n_1,\dots,n_l,n) \hookrightarrow \mathbb{P}^{N_1-1} \times \dots \times \mathbb{P}^{N_l-1}
$$
where $N_i = {n \choose n_i}$. This map is comes from the product of the morphisms determined by each $\mathcal{L}_i$ (see Theorem 3.2.13 of \cite{bcfkvs1}). By work of Gonciulea and Lakshmibai \cite{gs}, the Pl\"ucker embedding
$$
\phi: F(n_1,\dots,n_l,n) \hookrightarrow \mathbb{P}^{N_1-1} \times \dots \times \mathbb{P}^{N_l-1}
$$
gives a flat degeneration to the image of $\psi$. The divisors $C_i$ on $F(n_1,\dots,n_l,n)$ obtained by pulling back $\mathbb{P}^{N_1-1} \times \dots \times h_i \times \dots \times \mathbb{P}^{N_l-1}$ where $h_i$ is a generic hyperplane in $\mathbb{P}^{N_i-1}$ along $\phi$ form the Schubert basis of the Picard group of $F(n_1,\dots, n_k,n)$, and the ample cone is the interior of the cone generated over $\mathbb{R}_{\geq 0}$ by classes $C_i$ (see Proposition 1.4.1 of \cite{brion}). Furthermore, the anticanonical bundle divisor of $F(n_1,\dots,n_l,n)$ is given by
$$
-K_F = \sum_{i=1}^{l}m_i C_i.
$$
Here $m_i$ is the number of edges in the $i^{\mathrm{th}}$ roof of $\Gamma(n,n_1,\dots,n_l)$. We choose multi-degrees $\overline{d_i} = (d_i^{(1)},\dots, d_i^{(l)})$ for integers $1 \leq i \leq k$ so that $\sum_{i=1}^k d_i^{(j)} < m_j$. Let $\overline{d}$ denote this set of multidegrees. Then let $Z_{\overline{d_i}}$ be the intersection of $F(n_1,\dots,n_l,n)$ with a generic divisor of multi-degree $\overline{d_i}$ in $\mathbb{P}^{N_1-1} \times \dots \times\mathbb{P}^{N_l-1}$ under the emebedding $\phi$. The complete intersection $X_{\overline{d}}$ in $F(n_1,\dots,n_k,n)$ of the divisors $Z_{\overline{d_i}}$ is Fano, since by the adjunction formula, $-K_{X} =\left( \sum_{i=1}^l n_i C_i\right)|_X$ for $n_i  = m_i - \sum_{j=1}^k d_{j}^{(i)}>1$ is the restriction of a very ample divisor on $F(n_1,\dots,n_l,n)$. 

If we keep the divisors $Z_{\overline{d_i}}$ fixed and let $F(n_1,\dots,n_l,n)$ degenerate to $P(n_1,\dots,n_l,n)$, we obtain a natural degeneration of $X_{\overline{d}}$ to a generic complete intersection $X'_{\overline{d}}$ in $P(n_1,\dots,n_l,n)$ cut out by the vanishing locus of a non-degenerate global section of $\bigoplus_{i=1}^l \mathcal{O}(\sum_{j=1}^k \mathcal{L}_j^{d_j^{(i)}})$.

We may now associate $X'_{\overline{d}}$ to a nef partition of $\Delta(n_1,\dots,n_l,n)$. For each $\overline{d_i}$, choose a set $\mathcal{U}_{i,j}$ of $d_i^{(j)}$ vectors $\alpha \in \mathcal{R}_j$ in such a way that the sets $\mathcal{U}_{i,j}$ have pairwise empty intersection and so that no $\mathcal{U}_{i,j}$ contains an arrow $\alpha$ so that $h(\alpha)$ is a white vertex.

It is possible to choose sets this way since $\sum_{i=1}^k d_{i}^{(j)} < m_j$. Let $\mathcal{U}_i = \cup_{j=1}^l \mathcal{U}_{i,j}$. Thus we have divisors
$$
H_i = \sum_{\alpha \in \mathcal{U}_i} H_\alpha
$$
which are nef Cartier divisors on $P(n_1,\dots,n_l,n)$ linearly equivalent to the divisors $Z_{\overline{d_i}}$ restricted to $P(n_1,\dots,n_l,n)$ in $\mathbb{P}^{N_1-1} \times \dots \times \mathbb{P}^{N_l-1}$. Furthermore, $H_i$ correspond to a nef partition of $\Delta(n_1,\dots,n_l,n)$. Let $\mathcal{U}_{k+1} = (\cup_{i=1}^l \mathcal{R}_i )\setminus (\cup_{i=1}^k \mathcal{U}_i)$. Note $\mathcal{U}_{k+1}$ contains all arrows $\alpha$ with $h(\alpha)$ a white vertex. Then the sets
$$
E_ i := \bigcup_{\alpha \in \mathcal{U}_i} U(\alpha) 
$$
define a nef partition of $\Delta(n_1,\dots,n_l,n)$. We have the standard generating set of regular functions on $(\mathbb{C}^\times)^{D}$ written as $x_{m,n}$ associated to black vertices $v_{m,n}$ of $\Gamma(n_1,\dots,n_l,n)$. The monomial associated to an arrow $\alpha$ is  
$$
x^\alpha = \dfrac{x_{h(\alpha)}}{x_{t(\alpha)}}
$$
and we define the Givental Landau-Ginzburg mirror of $X_{\overline{d}}$ to be the complete intersection $X^\vee_{\overline{d}}$
$$
1 = \sum_{\alpha \in \mathcal{U}_i} a_\alpha x^\alpha
$$
for $1 \leq i \leq k$ equipped with superpotential
$$
w = \sum_{\alpha \in \mathcal{U}_{k+1}}a_\alpha x^\alpha.
$$
Here the coefficients $a_\alpha$ should be chosen so that they satisfy the so-called box equations and roof equations of Section 5.1 of \cite{bcfkvs1}.
\subsection{Associated amenable collections}

An element $\ell$ of $N = \Hom(M,\mathbb{Z})$ is determined by the number that it assigns to each generator of $M$. Since we have associated to each black vertex of $\Gamma(n_1,\dots,n_l,n)$ a generator $e_d$, and we have formally set $e_s$ to be the origin for $s \in S$ a white vertex, an element of $N$ just assigns to each black vertex of $\Gamma(n_1,\dots,n_l,n)$ some integer, and assigns the value $0$ uniformly to all white vertices. To the points in $\Delta(n_1,\dots,n_l,n)$ determined by edges $\alpha$ of $\Gamma(n_1,\dots,n_l,n)$, the linear operator $\ell$ assigns the number
$$
\ell(p_\alpha) = \ell(e_{h(\alpha)}) - \ell(e_{t(\alpha)}).
$$
Therefore, each $\ell \in N$ is simply a rule that assigns to each black vertex of $\Gamma(n_1,\dots,n_l,n)$ an integer so that the resulting value associated to the arrows in each $E_j$ is $(-1)$, takes non-negative values elsewhere, and takes the value $0$ on $E_k$ for $k < j$. Our task now is to choose carefully an amenable collection of vectors associated to a given nef partition. We will first describe this process for a single $\alpha$ in $\mathcal{R}_i$. There are two distinct cases to deal with:
\begin{enumerate}
\item The edge $\alpha$ is horizontal and $t(\alpha)$ and $h(\alpha)$ are black vertices.
\item The edge $\alpha$ is vertical.
\end{enumerate}
We treat these cases separately then combine them to produce the desired function. Let us take two white vertices of $\Gamma(n_1,\dots,n_l,n)$ located at points $(m_0,n_0)$ and $(m_1,n_1)$ so that there is no white vertex $(m_2,n_2)$ with $m_0\leq m_2\leq m_1$ and $n_1 \leq n_2 \leq n_0$, and let $\alpha$ be an edge in the roof between $(m_1,n_1)$ and $(m_2,n_2)$.

\begin{enumerate}
\item Let $\alpha$ be a vertical arrow so that $\alpha = (v_{m,n} \rightarrow v_{m,n-1})$ for $m_0 \leq m \leq m_1-1$ and $n_1 -1 \leq n \leq n_0$. Then we define the function $\ell_\alpha$ so that 
$$
\ell_\alpha(e_{(i,j)}) = \left\{
	\begin{array}{rl}
		-1  & \mbox{if }   i \leq n_1-1 \mbox{ and } j \leq  m-1\\
		 0  & \mbox{otherwise }
	\end{array}
\right.
$$
We can check the value of $\ell_\alpha$ on vertical arrows
$$
\ell_{\alpha}(e_{(i,j)})- \ell_\alpha(e_{(i,j-1)})  = \left\{
	\begin{array}{rl}
		-1  & \mbox{if } j = n  \\
		 0  & \mbox{otherwise }
	\end{array}
\right.
$$
and on horizontal arrows,
$$
\ell_{\alpha}(e_{(i,j)})- \ell_\alpha(e_{(i+1,j)})  = \left\{
	\begin{array}{rl}
		 1  & \mbox{if } i =  m_1-1 \\
		 0  & \mbox{otherwise }
	\end{array}
\right.
$$
Thus $\ell_\alpha$ takes value $(-1)$ only on elements of $U(\alpha)$ and takes positive values only at arrows $(v_{n_1-1,j} \rightarrow v_{n_1,j})$.

\item Now let us take some vector $\alpha \in \mathcal{R}_i$ so that $\alpha = (v_{m,n_0-1} \rightarrow v_{m+1,n_0-1})$ for $m_0 \leq m \leq m_1 - 2$. Define $\ell_\alpha$ on the basis $e_{(i,j)}$ so that 
$$
\ell_\alpha(e_{(i,j)}) = \left\{
	\begin{array}{rl}
		-1  & \mbox{if } m+1 \leq i \leq m_1-1  \mbox{ and } j \leq  n_0-1\\
		 0  & \mbox{otherwise }
	\end{array}
\right.
$$
Thus 
$$
\ell_{\alpha}(e_{(i,j)})- \ell_\alpha(e_{(i+1,j)})  = \left\{
	\begin{array}{rl}
		-1  & \mbox{if } i =  m \\
		 1  & \mbox{if } i =  m_1 - 1 \\
		 0  & \mbox{otherwise }
	\end{array}
\right.
$$
and for any vertical arrow,
$$
\ell_{\alpha}(e_{(i,j)})- \ell_\alpha(e_{(i1,j)})  = 0 
$$
\end{enumerate}

\begin{figure}
\begin{center}
\begin{tikzpicture}[scale = 1.1]

%Arrows
\draw[->] (0,4.95) -- node[right] {\tiny } (0,4.05);
\draw[ ->] (0,4) -- node[right] {\tiny } (0,3.05);
\draw[->] (0,3) -- node[right] {\tiny} (0,2.05);
\draw[->] (0,2) -- node[right] {\tiny} (0,1.05);
\draw[->] (0,1) -- node[right] {\tiny} (0,0.05);
\draw[->] (0,4) -- node[above] {\tiny } (0.95,4);
\draw[ultra thick,->] (1,4) -- node[below] {\tiny -1} node[above] {\tiny $\alpha_1$} (1.95,4);
\draw[->] (1,4) -- node[right] {\tiny} (1,3.05);
\draw[->] (1,3) -- node[right] {\tiny} (1,2.05);
\draw[->] (1,2) -- node[right] {\tiny} (1,1.05);
\draw[->] (1,1) -- node[right] {\tiny} (1,0.05);
\draw[->] (2,4) -- node[right] {\tiny } (2,3.05);
\draw[->] (2,3) -- node[right] {\tiny} (2,2.05);
\draw[->] (2,2) -- node[right] {\tiny} (2,1.05);
\draw[->] (2,1) -- node[right] {\tiny} (2,0.05);
\draw[->] (0,3) -- node[above] {\tiny} (0.95,3);
\draw[->] (1,3) -- node[below] {\tiny -1 } (1.95,3);
\draw[->] (0,2) -- node[above] {\tiny} (.95,2); 
\draw[->] (1,2) -- node[below] {\tiny -1} (1.95,2); 
\draw[->] (2,2) -- node[below]{\tiny 1} (2.95,2);
\draw[->] (0,1) -- node[above]{\tiny } (0.95,1); 
\draw[->] (1,1) -- node[below] {\tiny -1} (1.95,1); 
\draw[->] (2,1) -- node[below] {\tiny 1} (2.95,1); 
\draw[->] (3,1) -- node[above] {\tiny } (3.95,1);
\draw[->] (4,1) -- node[above] {\tiny } (4.95,1);
\draw[->] (3,1.95) -- node[right] {\tiny} (3,1.05); 
\draw[->] (3,1) -- node[right] {\tiny } (3,0.05);
\draw[->] (0,0) -- node[above]{ \tiny } (.95,0); 
\draw[->] (1,0) -- node[below] {\tiny -1 } (1.95,0); 
\draw[->] (2,0) -- node[below] {\tiny 1} (2.95,0); 
\draw[->] (3,0) -- node[above] {\tiny} (3.95,0);
\draw[->] (4,0) -- node[above] {\tiny } (4.95,0);
\draw[->] (5,0) -- node[above] {\tiny } (5.95,0);
\draw[->] (4,1) -- node[right] {\tiny } (4,0.05); 
\draw[->] (5,1) -- node[right] {\tiny } (5,0.05);

% Nodes
\draw(0,0) node[circle,fill,inner sep=1pt, label = below left:\tiny ](){};
\draw(0,1) node[circle,fill,inner sep=1pt, label = below left:\tiny ](){};
\draw(0,2) node[circle,fill,inner sep=1pt, label = below left: \tiny](){};
\draw(0,3) node[circle,fill,inner sep=1pt, label = below left: \tiny](){};
\draw(0,4) node[circle,fill,inner sep=1pt, label = below left: \tiny](){};

\draw(1,0) node[circle,fill,inner sep=1pt, label = below left : \tiny](){};
\draw(1,1) node[circle,fill,inner sep=1pt, label = below left : \tiny](){};
\draw(1,2) node[circle,fill,inner sep=1pt, label = below left : \tiny](){};
\draw(1,3) node[circle,fill,inner sep=1pt, label = below left : \tiny](){};
\draw(1,4) node[circle,fill,inner sep=1pt, label = below left : \tiny](){};

\draw(2,0) node[circle,fill,inner sep=1pt, label = above right : \tiny -1](){};
\draw(2,1) node[circle,fill,inner sep=1pt, label = above right : \tiny-1](){};
\draw(2,2) node[circle,fill,inner sep=1pt, label = above right : \tiny-1](){};
\draw(2,3) node[circle,fill,inner sep=1pt, label = above right : \tiny-1](){};
\draw(2,4) node[circle,fill,inner sep=1pt, label = above right : \tiny-1](){};

\draw(3,1) node[circle,fill,inner sep=1pt, label = below left : \tiny](){};
\draw(3,0) node[circle,fill,inner sep=1pt, label = below left : \tiny](){};

\draw(4,1) node[circle,fill,inner sep=1pt,label = below left : \tiny ](){};
\draw(4,0) node[circle,fill,inner sep=1pt, label = below left : \tiny ](){};

\draw(5,1) node[circle,fill,inner sep=1pt, label = below left : \tiny](){};
\draw(5,0) node[circle,fill,inner sep=1pt, label = below left : \tiny](){};

\draw(6,0) node[circle,inner sep=1pt, draw, label = below left :\tiny ](){};

\draw(0,5) node[circle,inner sep=1pt, draw, label = below left:{\tiny}](){};
\draw(3,2) node[circle,inner sep=1pt, draw, label = below left: \tiny ](){};

\end{tikzpicture}
\begin{tikzpicture}[scale = 1.1]

%Arrows
\draw[->] (0,4.95) -- node[right] {\tiny } (0,4.05);
\draw[ ->] (0,4) -- node[right] {\tiny } (0,3.05);
\draw[->] (0,3) -- node[left] {\tiny-1 } (0,2.05);
\draw[->] (0,2) -- node[right] {\tiny} (0,1.05);
\draw[->] (0,1) -- node[right] {\tiny} (0,0.05);
\draw[->] (0,4) -- node[above] {\tiny } (0.95,4);
\draw[ ->] (1,4) -- node[above] {\tiny } (1.95,4);
\draw[->] (1,4) -- node[right] {\tiny} (1,3.05);
\draw[->] (1,3) -- node[left] {\tiny -1} (1,2.05);
\draw[->] (1,2) -- node[right] {\tiny} (1,1.05);
\draw[->] (1,1) -- node[right] {\tiny} (1,0.05);
\draw[ ->] (2,4) -- node[right] {\tiny } (2,3.05);
\draw[ultra thick, ->] (2,3) -- node[left] {\tiny -1 } node[right] {\tiny $\alpha_2$}  (2,2.05);
\draw[->] (2,2) -- node[right] {\tiny} (2,1.05);
\draw[->] (2,1) -- node[right] {\tiny} (2,0.05);
\draw[->] (0,3) -- node[above] {\tiny} (0.95,3);
\draw[->] (1,3) -- node[above] {\tiny} (1.95,3);
\draw[->] (0,2) -- node[above] {\tiny} (.95,2); 
\draw[->] (1,2) -- node[above] {\tiny} (1.95,2); 
\draw[->] (2,2) -- node[above]{\tiny 1} (2.95,2);
\draw[->] (0,1) -- node[above]{\tiny } (0.95,1); 
\draw[->] (1,1) -- node[above] {\tiny } (1.95,1); 
\draw[->] (2,1) -- node[above] {\tiny 1} (2.95,1); 
\draw[->] (3,1) -- node[above] {\tiny } (3.95,1);
\draw[->] (4,1) -- node[above] {\tiny } (4.95,1);
\draw[->] (3,1.95) -- node[right] {\tiny} (3,1.05); 
\draw[->] (3,1) -- node[right] {\tiny } (3,0.05);
\draw[->] (0,0) -- node[above]{ \tiny } (.95,0); 
\draw[->] (1,0) -- node[above] {\tiny } (1.95,0); 
\draw[->] (2,0) -- node[above] {\tiny 1} (2.95,0); 
\draw[->] (3,0) -- node[above] {\tiny} (3.95,0);
\draw[->] (4,0) -- node[above] {\tiny } (4.95,0);
\draw[->] (5,0) -- node[above] {\tiny } (5.95,0);
\draw[->] (4,1) -- node[right] {\tiny } (4,0.05); 
\draw[->] (5,1) -- node[right] {\tiny } (5,0.05);

% Nodes
\draw(0,0) node[circle,fill,inner sep=1pt, label = below left:\tiny -1](){};
\draw(0,1) node[circle,fill,inner sep=1pt, label = below left:\tiny -1](){};
\draw(0,2) node[circle,fill,inner sep=1pt, label = below left: \tiny-1](){};
\draw(0,3) node[circle,fill,inner sep=1pt, label = below left: \tiny](){};
\draw(0,4) node[circle,fill,inner sep=1pt, label = below left: \tiny](){};

\draw(1,0) node[circle,fill,inner sep=1pt, label = below left : \tiny-1 ](){};
\draw(1,1) node[circle,fill,inner sep=1pt, label = below left : \tiny-1 ](){};
\draw(1,2) node[circle,fill,inner sep=1pt, label = below left : \tiny-1 ](){};
\draw(1,3) node[circle,fill,inner sep=1pt, label = below left : \tiny](){};
\draw(1,4) node[circle,fill,inner sep=1pt, label = below left : \tiny](){};

\draw(2,0) node[circle,fill,inner sep=1pt, label = below left : \tiny -1](){};
\draw(2,1) node[circle,fill,inner sep=1pt, label = below left : \tiny-1](){};
\draw(2,2) node[circle,fill,inner sep=1pt, label = below left : \tiny-1](){};
\draw(2,3) node[circle,fill,inner sep=1pt, label = below left : \tiny](){};
\draw(2,4) node[circle,fill,inner sep=1pt, label = below left : \tiny](){};

\draw(3,1) node[circle,fill,inner sep=1pt, label = below left : \tiny](){};
\draw(3,0) node[circle,fill,inner sep=1pt, label = below left : \tiny](){};

\draw(4,1) node[circle,fill,inner sep=1pt,label = below left : \tiny ](){};
\draw(4,0) node[circle,fill,inner sep=1pt, label = below left : \tiny ](){};

\draw(5,1) node[circle,fill,inner sep=1pt, label = below left : \tiny](){};
\draw(5,0) node[circle,fill,inner sep=1pt, label = below left : \tiny](){};

\draw(6,0) node[circle,inner sep=1pt, draw, label = below left :\tiny ](){};

\draw(0,5) node[circle,inner sep=1pt, draw, label = below left:{\tiny}](){};
\draw(3,2) node[circle,inner sep=1pt, draw, label = below left: \tiny ](){};

\end{tikzpicture}
\end{center}
\label{fig:fn}
\caption{Functions $\ell_{\alpha_i}$ associated to a horizontal and vertical arrows $\alpha_1,\alpha_2 \in \mathcal{R}_1$ respectively. Vertices and arrows which have not been assigned numbers correspond to vertices and arrows to which $\ell_\alpha$ assigns the number $0$.}
\label{fig:fn1}
\end{figure}
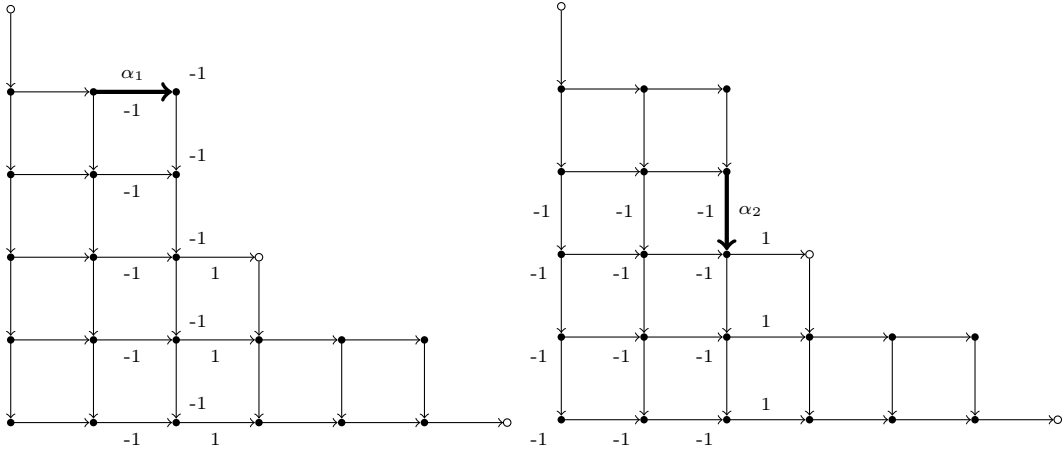

Examples of $\ell_\alpha$ for both vertical and horizontal arrows $\alpha$ are shown in Figure \ref{fig:fn1}. Thus we have chosen $\ell_\alpha \in N$ for each $\alpha \in \mathcal{R}_i$ so that $h(\alpha)$ is not a white vertex, in such a way that $\ell_\alpha$ takes value $(-1)$ only at arrows in $U(\alpha)$ and which takes positive values only at horizontal arrows $(v_{m_0-1,j} \rightarrow v_{m_0,j})$. Thus for any arrows $\alpha_1 \in \mathcal{R}_i$ and $\alpha_2\in \mathcal{R}_j$ for which $h(\alpha_i)$ is not a white vertex, we have $\ell_{\alpha_1}(\alpha) = 0$ for all $\alpha \in U(\alpha_2)$ and $\ell_{\alpha_2}(\alpha) = 0$ for all $\alpha \in U(\alpha_1)$.

Now let us choose some $(k+1)$-partite nef partition of $\Delta(n_1,\dots,n_l,n)$ given by multidegrees $\overline{d}_i = (d_{1}^{(i)},\dots, d_{r}^{(i)})$ so that $\sum_{i=1}^k d_{j}^{(i)} < m_i$. Then, as in Section \ref{sect:Pn1}, we may choose disjoint collections $\mathcal{U}_j$ of vectors in the union of all roofs $\cup_{i=1}^r \mathcal{R}_j$ so that $\mathcal{U}_j \cap \mathcal{R}_i$ is of size $d_{i}^{(j)}$ and so that for all $\mathcal{U}_j$ there is no $\alpha \in \mathcal{U}_j$ for which $h(\alpha)$ is a white vertex. Define 
$$
\ell_{\mathcal{U}_j} = \sum_{\alpha \in \mathcal{U}_j} \ell_\alpha.
$$
and let $E_1,\dots,E_{k+1}$ be the nef partition described in Section \ref{sect:Pn1} associated to the sets $\mathcal{U}_i$.
\begin{proposition}
\label{prop:amenflag}
If we have a $(k+1)$-partite nef partition as described in the preceding paragraph, then the collection of vectors $V = \{ \ell_{\mathcal{U}_1},\dots, \ell_{\mathcal{U}_k}\}$ forms an amenable collection of vectors subordinate to the chosen $(k+1)$-partite nef partition.
\end{proposition}
\begin{proof}
It is enough to show that $\ell_{\mathcal{U}_i}(\beta) = 0$ for any $\beta \in U(\alpha)$ for $\alpha \in \mathcal{U}_j$ and $j \neq i$. However, this follows easily from the fact that each $\ell_\alpha$ takes the value $(-1)$ at $\beta \in U(\alpha)$, positive values on arrows in $U(\delta)$ with $h(\delta)$ a white vertex and 0 otherwise. Thus $\ell_{\mathcal{U}_j}$ takes values $(-1)$ only at arrows $\beta \in U(\alpha)$ for $\alpha \in \mathcal{U}_j$ and positive values on arrows in $U(\delta)$ with $h(\delta)$ a white vertex and 0 otherwise. We have that $U(\alpha) \cap U(\delta) = \emptyset$ if $\alpha \neq \delta$, thus since $\mathcal{U}_i$ contains no arrow $\delta$ with $h(\delta)$ a white vertex, $\ell_{\mathcal{U}_i}(\beta) = 0$ if $\alpha \in U(\alpha)$ and $\alpha \in \mathcal{U}_j$ with $j \neq i$. 
\end{proof}
Therefore, we may conclude, following Theorem \ref{thm:main}, that

\begin{theorem}
\label{thm:flags}
Let $X'_{\overline{d}}$ be a Fano complete intersection in $P(n_1,\dots,n_l,n)$ determined by a set of multi-degrees $\overline{d}$. Then $X'_{\overline{d}}$ admits a degeneration to a weak Fano toric variety $X_\Sigma$ with at worst Gorenstein singularities. Furthermore, the Givental Landau-Ginzburg model of $X'_{\overline{d}}$ admits a torus map $\phi_{\overline{d}}$ so that the pullback of the superpotential $w$ along $\phi_{\overline{d}}$ is a Laurent polynomial with Newton polytope $\Delta$ so that $X_\Delta = X_\Sigma$.
\end{theorem}

\begin{ex}\label{ex:F125}
We conclude with a non-trivial example of our method at work. Let us take the partial flag manifold $F(1,2,5)$, and we will compute the Laurent polynomial associated to a Fano hypersurface in this Flag manifold. First, we have variables
$x_{0,1},x_{0,0},x_{1,1},x_{1,0}, x_{2,1},x_{2,0}$ and $x_{3,0}$, and we choose the nef partition of $\Delta(1,2,5)$ associated to the roof-paths of length $3$ and $1$ in each block (in other words the multi-degree $\overline{d}$ is just $(3,1)$). This nef partition corresponds to the following Givental Landau-Ginzburg model,
$$
1= x_{0,1} + \dfrac{x_{0,0}}{x_{0,1}} + \dfrac{x_{1,0}}{x_{1,1}} + \dfrac{x_{2,0}}{x_{2,1}} + x_{3,0} + \dfrac{x_{1,1}}{x_{0,1}} + \dfrac{x_{2,1}}{x_{1,1}} + \dfrac{x_{1,0}}{x_{0,0}} + \dfrac{x_{2,0}}{x_{1,0}} 
$$
equipped with potential
$$
w = \dfrac{1}{x_{2,1}} + \dfrac{1}{x_{3,0}} + \dfrac{x_{3,0}}{x_{2,0}}. 
$$
The method described in Proposition \ref{prop:amenflag} produces an amenable collection with only one element, which is given by
$$
v = -e_{(0,1)}^* -2e_{(0,0)}^* -2e^*_{(1,1)} - 3e^*_{(1,0)} - 4e_{(2,0)}^* - e^*_{(3,0)} - 3 e_{(2,1)}^*
$$
which may be completed to a basis if we let $v_2 = e_{(0,0)}^*$,  $v_3 =e^*_{(1,1)}$, $v_4 = e^*_{(1,0)}$, $v_5 = e_{(2,0)}^*$  $v_6 = e^*_{(3,0)}$ and $v_7=e_{(2,1)}^*$. Then in terms of this basis, the Givental Landau-Ginzburg model looks like
$$
1 = \dfrac{1}{y_1} +\dfrac{y_2}{y_1} + \dfrac{y_4}{y_1y_3} + \dfrac{y_5}{y_1y_7} + \dfrac{y_6}{y_1} + \dfrac{y_3}{y_1} + \dfrac{y_7}{y_1y_3} + \dfrac{y_4}{y_1y_2} + \dfrac{y_5}{y_1y_4} 
$$
with potential
$$
w = \dfrac{y_1^3}{y_7} + \dfrac{y_1}{y_6} + \dfrac{y_1^3y_6}{y_5}.
$$
Eliminating $y_1$ from the first equation, we obtain
$$
y_1 = 1 + y_2 + \dfrac{y_4}{y_3} + \dfrac{y_5}{y_7} + y_6 +y_3+ \dfrac{y_7}{y_3} + \dfrac{y_4}{y_2} + \dfrac{y_5}{y_4} 
$$
and thus
\tiny
$$
w =\left(1 + y_2 + \dfrac{y_4}{y_3} + \dfrac{y_5}{y_7} + y_6 +y_3+ \dfrac{y_7}{y_3} + \dfrac{y_4}{y_2} + \dfrac{y_5}{y_4}\right) \left(\dfrac{1}{y_6} + \left(1 + y_2 + \dfrac{y_4}{y_3} + \dfrac{y_5}{y_7} + y_6 +y_3+ \dfrac{y_7}{y_3} + \dfrac{y_4}{y_2} + \dfrac{y_5}{y_4}\right)^2\left(\dfrac{y_6}{y_5} + \dfrac{1}{y_7}\right)\right)
$$
\end{ex}

\section{Further applications}
\label{sect:apply} 

Recently, Coates, Kasprzyk and Prince \cite{ckp} have given a reasonably general method of turning a Givental Landau-Ginzburg model into a Laurent polynomial under specific conditions. We will show that all of their Laurent polynomials are cases of Theorem \ref{thm:almost}, and that all of the Laurent polynomials of Coates, Kasprzyk and Prince come from toric degenerations. We will also comment on the extent to which we recover results of Ilten Lewis and Przyjalkowski \cite{ilp11}, and mention how our results relate to geometric transitions of toric complete intersection Calabi-Yau varieties.

\subsection{The Przyjalkowski method}
\label{sect:ckp}

Here we recall the Przyjalkowski method as described by Coates, Kasprzyk and Prince in \cite{ckp} and show that their construction can be recast in terms of amenable toric degenerations. We will conclude that if the Przyjalkowski method is applied when $Y_\Delta$ is a Fano toric variety, then results of Section \ref{sect:polytopes} imply that all of the Laurent polynomials obtained in \cite{ckp} correspond to amenable toric degenerations of the complete intersection $X$. 

Begin with a smooth toric Fano variety $Y_\Delta$ obtained from a reflexive polytope $\Delta \subseteq M\otimes_\mathbb{Z} \mathbb{R}$ with $M$ a lattice of rank $m$. Then we have an exact sequence
\begin{equation}
\label{eq:ses}
0 \rightarrow \Hom(M,\mathbb{Z}) \rightarrow \mathbb{Z}^{N} \xrightarrow{(m_{ij})} \Pic(Y_\Delta) \rightarrow 0
\end{equation}
where the vertices of $\Delta$ are given an ordering and identified with elements of the set $\{1,\dots, N\}$ and where $\Pic(Y_\Delta)$ is the Cartier divisor class group of $Y_\Delta$. We make the following choices: let $E$ be a a subset of $\{1,\dots,N\}$ corresponding to a set of torus invariant divisors which generate $\Pic(Y_\Delta)$ and let $S_1,\dots, S_k$ be disjoint sets subsets of $\{1,\dots,N\}$ whose corresponding divisors may be expressed as non-negative linear combinations in elements of divisors corresponding to elements of $E$. Assume that each $S_i$ is disjoint from $E$. Torus invariant divisors of $Y_\Delta$ correspond to vertices of $\Delta$. The method of Hori-Vafa \cite{hv} for producing Landau-Ginzburg models for $X$ is then applied. Take variables $x_i$ for $1 \leq i \leq N$, which can be though of as coordinates on the torus $(\mathbb{C}^\times)^N$, and impose relations
$$
q_\ell = \prod_{j=1}^m x_j^{m_{\ell j}} 
$$
for each $\ell \in E$ and $q_\ell$ a variable in $\mathbb{C}^\times$, and equip the associated toric subvariety of $(\mathbb{C}^\times)^N$ with the superpotential
$$
w = \sum_{i=1}^N x_i
$$
By assumption, we have that elements of $E$ form a basis of $\Pic(Y_\Delta)$. Therefore, the matrix $(m_{ij})$ can be written as the identity matrix when restricted to the subspace of $\mathbb{Z}^N$ spanned by elements in $E$. Since the sequence in Equation \ref{eq:ses} is exact the elements $\{e_1,\dots, e_n\}$ of $E$ are part of a basis $\{e_1,\dots, e_n, u_{n+1},\dots, u_N\}$ of $\mathbb{Z}^N$. In this basis, we have 
$$
q_\ell = \prod_{j=1}^m x_j^{m_{\ell j}} = x_\ell \prod_{j=1, i \neq \ell}^N x_j^{m_{\ell j }}
$$
and thus we obtain the relations
$$
x_{\ell} = \dfrac{q_{\ell}}{\prod_{j=1, j \neq \ell}^N x_j^{m_{\ell j }}}.
$$
The superpotential for $Y_\Delta$ then becomes
\begin{equation}
\label{eq:potential}
w = \sum_{\ell \in E} \left( \dfrac{q_{\ell}}{\prod_{j=1, j \neq \ell}^N x_j^{m_{\ell j }}}\right) + \sum_{i \notin E} x_i
\end{equation}
The monomials in $w$ correspond to the vertices of $\Delta$, and we have eliminated variables corresponding to elements of $E$. Since $E$ has cardinality equal to $\rank(\Pic(Y_\Delta))$, the superpotential $w$ is expressed in terms of $n$ variables. All values $m_{ij}$ involved in the expression above are non-negative if $j \in S_i$ for some $1\leq i\leq k$, since we have chosen $S_1,\dots, S_k$ to be non-negative linear combinations in $\Pic(Y_\Delta)$ of elements in $E$.

The Givental Landau-Ginzburg model of $X$ is then given by the subspace $X^\vee$ of $(\mathbb{C}^\times)^{N}$ cut out by equations
$$
1 = \sum_{j\in S_i} x_j \text{ for } 1 \leq i \leq k.
$$
Equipped with the superpotential obtained by restricting $w$ to $X^\vee$. This agrees with the notion of Givental Landau-Ginzburg model presented in Section \ref{sect:general} up to a translation by the constant $k$.

At this point, the authors of \cite{ckp} choose an element $s_i \in S_i$ for each $1 \leq i \leq k$ and then make the variable substitutions for each $\ell \in S_i$
$$
x_\ell = \left\{
	\begin{array}{rl}
		& \dfrac{y_\ell}{1+ \sum_{j \in S_i, j \neq s_i} y_j} \text{ if } \ell \neq s_i \\ 
              & \dfrac{1}{1+ \sum_{j \in S_i, j \neq s_i } y_j} \text{ if } \ell= s_i 
	\end{array}
\right. 
$$
These expressions for $x_\ell$ in terms of $y_j$ then parametrize the hypersurfaces defined by the equations
$$
1 = \sum_{j \in S_i} x_j.
$$
Since all $m_{ij}$ in Equation \ref{eq:potential} are non-negative for $j \in \cup_{i=1}^k S_i$, substitution turns $w$ into a Laurent polynomial expressed in terms of $n-k$ variables, $y_\ell$ for $\ell \in \cup_{i=1}^k S_i$ and $x_j$ for $j \in \{1,\dots,N \} \setminus (\cup_{i=1}^k S_i \cup E)$.

\subsection{Associated amenable collections}

Now we rephrase Przyjalkowski's method in terms of our discussion in Section \ref{sect:general}. Since the monomials of $w$ correspond to vertices of $\Delta$, the conditions on $S_1,\dots,S_k$ and $E$ restrict $\Delta$ so that we may choose $m$ vertices of $\Delta$ which correspond to a spanning set $\{e_1,\dots, e_n\}$ of $M$. Then the remaining vertices of $\Delta$, and $S_1,\dots,S_k$ correspond to subsets of this spanning set. Furthermore, the insistence on positivity of elements of $S_1,\dots,S_k$ in terms of elements of $E$ means that every vertex of $E$ must be a sum $-\sum_{j=1}^n m_{i,j}e_j$ so that $m_{i,j}$ is positive for $j$ corresponding to an element of $\cup_{i=1}^kS_i$. Thus $e_1,\dots, e_n$ must actually span a maximal facet of $\Delta$.

In other words, we have an $n$-dimensional polytope $\Delta$ with simplicial face with vertices $\{e_1,\dots,e_n\}$ a generating set for $M$ so that $Y_\Delta$ is a smooth Fano toric variety. We have now chosen a partition of $\Delta[0]$ so that $E_1,\dots,E_k$ correspond to the vertices to which elements of $S_1,\dots,S_k$ correspond and are thus composed of disjoint subsets of $\{e_1,\dots,e_n\}$. The set $E_{k+1}$ is simply the complement $\Delta[0] \setminus \cup_{i=1}^k E_i$. Furthermore, we have chosen $E_i$ so that elements of $u\in E_{k+1}$ are written as $u = -\sum_{j=1}^n m_{i,j}e_j$ and $m_{i,j} \leq 0$ if $e_j \in E_i$ for $1 \leq i \leq k$.

\begin{proposition}
The sets $E_1,\dots, E_k$ and $E_{k+1}$ form a nef partition of $\Delta$.
\end{proposition} 
\begin{proof}
By definition, this is a partition of vertices of $\Delta$. It remains to show the existence of convex $\Sigma_\Delta$-piecewise linear functions compatible with this partition, but this follows from the assumption that $Y_\Delta$ is a smooth Fano toric variety, hence all irreducible and reduced torus invariant Weil divisors in $Y_\Delta$ are nef and Cartier.
\end{proof}

The problem is then to show that there are $v_i$ in the lattice $N = \Hom(M,\mathbb{Z})$ so that the method of Section \ref{sect:general} recovers the Laurent polynomial of \cite{ckp}.

\begin{proposition}
\label{prop:ckp}
Let $E_1,\dots,E_{k+1}$ be a nef partition chosen as above. Then there is an amenable collection of vectors $V$ subordinate to this nef partition of $\Delta$ so that the resulting Laurent polynomial is the same as the Laurent polynomial obtained by the Przyjalkowski method.
\end{proposition}
\begin{proof}
Let $e_1^*, \dots, e_m^*$ be the basis of $N$ dual to $e_1,\dots, e_d$
$$
v_i = -\sum_{e_j \in E_i} e_j^*.
$$
This choice of $v_i$ then satisfies $\langle v_i, e_j \rangle = -1$ if $e_j \in E_i$,$\langle v_i, e_j \rangle = 0$ if $e_j \in E_j$ for $j \neq i, k+1$, and $\langle v_i, \rho \rangle \geq 0$ for $\rho \in E_{k+1}$. Thus $v_1,\dots, v_k$ forms an amenable set of vectors. To see that this amenable collection of vectors recovers the Laurent polynomial coming from the Przyjalkowski method, we must choose vectors $v_{k+1},\dots, v_n \in N$ so that $v_1,\dots, v_n$ form a basis of $N$. Here we use the choice of $s_i \in S_i$. Each $s_i$ corresponds to some vertex of $\Delta$ represented by a basis vector of $M$ which we may assume is given by $e_i$ up to re-ordering of the basis of $M$. It is then easy to check that $\{v_1,\dots, v_k \} \cup \{v_{k+1} = e^*_{k+1},\dots,v_n = e^*_n\}$ form a basis for the lattice $N$. In terms of this basis, we have
$$
1 = \sum_{\rho \in E_i} \left(\prod_{j=1}^{k+1} x_i^{\langle v_j, \rho \rangle}\right) =\dfrac{1}{x_{i}}+ \sum_{e_j \in E_i, j \neq i} \dfrac{x_j}{x_{i}}
$$
and thus we have a torus map
$$
\phi_V:(\mathbb{C}^\times)^{n-k} \dashrightarrow X^\vee
$$
parametrizing $X^\vee$ given by variable assignment
$$
x_i =  \left\{
	\begin{array}{rl}
		& 1 + \sum_{e_j \in E_i, j \neq i}y_j \text{ if } 1\leq i \leq k \\ 
              &  y_i \text{ otherwise}
	\end{array}
\right. 
$$
This is expressed in torus coordinates which are dual to the basis $v_1,\dots, v_n$. This is, of course different from the map used in the Przyjalkowski method, but only because we have changed to a basis dual to $v_1,\dots,v_n$ and not the basis $e^*_1,\dots,e^*_n$. Changing basis so that we return to the standard basis with which we began, we must make the toric change of variables on $(\mathbb{C}^\times)^n$
$$
x_j = \left\{
	\begin{array}{rl}
		& \dfrac{z_j}{z_i} \text{ if } e_j \in E_i \text{ for } 1 \leq i \leq k \\ 
              & \dfrac{1}{z_j} \text{ if } 1 \leq j \leq k \\
              &  z_j \text{ otherwise}
	\end{array}
\right. 
$$
In these coordinates, $\phi$ is written as
$$
z_j = \left\{
	\begin{array}{rl}
		& \dfrac{y_j}{1 + \sum_{e_j\in E_i, j \neq i}y_j} \text{ if } e_j \in E_i \text{ for } 1 \leq i \leq k \\ 
              & \dfrac{1}{1 + \sum_{e_j\in E_i, j \neq i}y_j} \text{ if } 1 \leq j \leq k \\
              &  y_j \text{ otherwise}
	\end{array}
\right.
$$
which is precisely the embedding given by the Przyjalkowski method.
\end{proof}

Of course, as a corollary to this, Theorem \ref{thm:main} allows us to conclude that the Przyjalkowski method produces toric degenerations of the complete intersection with which we began.
\begin{theorem}
\label{thm:ckp}
Let $Y_\Delta$ be a smooth toric Fano manifold and let $X$ be a Fano complete intersection in $Y$. If the Givental Landau-Ginzburg model of $X$ becomes a Laurent polynomial with Newton polytope $\Delta'$ by the Przyjalkowski method, then $X$ degenerates to the toric variety $X_{\Delta'}$.
\end{theorem}

\subsection{Relation to \cite{ilp11}}
Perhaps it now should be mentioned how this work relates to work of Przyjalkowski \cite{prz} and Ilten, Lewis and Przyjalkowski \cite{ilp11}. In their situation, they begin with a smooth complete intersection Fano variety $X$ in a weighted projective space $\mathbb{WP}(w_0,\dots,w_n)$. By Remark 8 of \cite{prz}, we may assume that $w_0=1$, and hence the polytope $\Delta$ defining $\mathbb{WP}(1,\dots,w_n)$, has vertices given by the points $e_1,\dots,e_n$ and $- \sum_{i=1}^n w_i e_i$ for $\{e_1,\dots,e_n\}$ a basis of $M$. Then the Przyjalkowski method may be applied, essentially verbatim, letting $S_1,\dots, S_k$ correspond to subsets of the vertices $\{e_1,\dots,e_n\}$ and $E = \{ -\sum_{i=1}^n w_ie_i\}$.

Then the amenable collection constructed in the proof of Proposition \ref{prop:ckp} is given by 
$$
v_i = -\sum_{j \in S_i} e_j^*
$$
produces a Laurent polynomial associated to the Givental Landau-Ginzburg model identical to those constructed by Przyjalkowski in \cite{prz}, up to a toric change of basis. Proof of this is essentially identical to the proof of Proposition \ref{prop:ckp}. Since Przyjalkowski assumes that the divisors of $\mathbb{WP}(1,w_1,\dots,w_n)$ which cut out $X$ are Cartier, we have that $X$ is associated to a $\mathbb{Q}$-nef partition $E_1,\dots,E_{k+1}$ where $E_1,\dots, E_k$ are Cartier. This allows us to apply Theorem \ref{thm:main} to show
\begin{proposition}
\label{prop:ilp+h}
There is a degeneration of each smooth Fano weighted projective complete intersection to a toric variety $X_\Sigma$ so that the convex hull of the ray generators of $\Sigma$ is a polytope equal to the Newton polytope of the Laurent polynomial associated to $X$ in \cite{prz}.
\end{proposition}

This is a weaker version of the theorem proved in \cite{ilp11}.

\begin{theorem}[\cite{ilp11} Theorem 2.2]
\label{thm:ilp}
Let $\Delta_f$ be the Newton polytope of the Laurent polynomial associated to a smooth Fano weighted projective complete intersection $X$ in \cite{prz}. Then there is a degeneration of $X$ to $\widetilde{\mathbb{P}}(\Delta_f)$, as defined in Section 1.1 of \cite{ilp11}.
\end{theorem}
The difference between these two statements is that Proposition \ref{prop:ilp+h} shows that $X$ degenerates to a toric variety which is possibly a toric blow-up of the variety to which Theorem \ref{thm:ilp} shows that $X$ degenerates.

\subsection{Geometric transitions of Calabi-Yau varieties}
\label{sect:transition}
Readers interested in compact Calabi-Yau varieties, should note that we may reinterpret the work in Section \ref{sect:general} as a general description of geometric transitions of toric complete intersection Calabi-Yau varieties. 

We note that there is a reinterpretation of the map $\phi_V : (\mathbb{C}^\times)^{n-k} \dashrightarrow (\mathbb{C}^\times)^n$ as a section of the toric morphism $\pi_V: (\mathbb{C}^\times)^n \rightarrow (\mathbb{C}^\times)^{n-k}$ given by
$$
(x_1,\dots,x_n) \mapsto (x_{k+1},\dots, x_n).
$$
which sends the subscheme of $X^\vee$ cut out by the equations $w - \lambda$ for some complex value $\lambda$ to the subscheme of $(\mathbb{C}^\times)^{n-k}$ cut out by the vanishing locus of $\phi_V^*w - \lambda$ in $(\mathbb{C}^\times)^{n-k}$. Thus we obtain a birational map between the fibers of $w$, and fibers of the Laurent polynomial $\phi_V^*w$ which may be compactified to anticanonical hypersurfaces in $X_{(\Delta_{\phi_V^*w})^\circ}$.

Note that if $E_1,\dots, E_{k+1}$ is a nef partition of a Fano toric variety determined by a reflexive polytope $\Delta$, then $E_1,\dots,E_{k+1}$ determine a Calabi-Yau complete intersection $Z$ in $Y_\Delta$, which is precisely an anticanonical hypersurface in the complete intersection quasi-Fano variety $X$ determined by $E_1,\dots, E_k$. According to Batyrev and Borisov \cite{batbor}, there is a reflexive polytope $\nabla$ determined by $E_1,\dots,E_{k+1}$ and a dual $(k+1)$-partite nef partition of $Y_\nabla$ which determines a complete intersection Calabi-Yau variety $Z^\vee$ which is called the Batyrev-Borisov mirror dual of $Z$. 

It is well known \cite{Givental} that the {\it fibers} of the Givental Landau-Ginzburg model of $X$ may be compactified to complete intersections in $Y_\nabla$, and that these compactified fibers are the Batyrev-Borisov mirror dual to anticanonical hypersurfaces $Z$ in $X$. 

Now if we degenerate the homogeneous equations in the coordinate ring of $Y_\Delta$ defining $X$ to equations defining some toric variety $X_{\Delta_V}$, then we obtain simultaneous degenerations of anticanonical hypersurfaces $Z$ in $X$ to anticanonical hypersurfaces $Z'$ of $X_{\Delta_V}$. In general, anticanonical hypersurfaces of $X_{\Delta_V}$ are more singular than anticanonical hypersurfaces of $X$.

Classically, mirror symmetry predicts that there is a contraction of $Z^\vee \rightarrow (Z')^\vee$ which is mirror dual to the degeneration $Z \rightsquigarrow Z'$ where $Z'$ and $(Z')^\vee$ are mirror dual. Since $Z'$ is a toric hypersurface, the contracted variety $(Z')^\vee$ should be a hypersurface in the toric variety $X_{(\Delta_V)^\circ}$.

We deduce the following:
\begin{theorem}
\label{prop:CY}
Let $Z$ be an anticanonical hypersurface in a quasi-Fano complete intersection $X$ in a toric Fano variety $Y_\Delta$ determined by a nef partition $E_1,\dots,E_{k+1}$ and so that $E_1,\dots,E_k$ determines the quasi-Fano variety $X$. Assume there is an amenable collection of vectors subordinate to the nef partition $E_1,\dots,E_{k+1}$ which determines an amenable degeneration $X \rightsquigarrow X_{\Sigma_V}$ where the convex hull of the ray generators of $\Sigma_V$ is a reflexive polytope $\Delta_V$. Then $Z$ degenerates to a hypersurface in $X_{\Delta_V}$, and there is a mirror birational map from $Z^\vee$ to an anticanonical hypersurface in $X_{(\Delta_V)^\circ}$
\end{theorem}

Note that this is just a birational map, not necessarily a birational contraction. In work of Fredrickson \cite{fred}, it is shown that an associated birational contraction exists in several cases, once one performs appropriate partial resolutions of singularities on both $Z^\vee$ and $(Z')^\vee$. In \cite{mav2}, Mavlyutov showed that any toric variety $X_\Delta$ with a fixed  Minkowski decomposition of $\Delta^\circ$ can be embedded in a Fano toric variety $Y$ determined by the Cayley cone associated to the given Minkowski decomposition, and that anticanonical hypersurfaces in $X_\Delta$ can be deformed to nondegenerate nef complete intersections in $Y$. He then showed that a mirror contraction exists if the degeneration of $X$ to $X_\Delta$ is obtained in this way.

\bibliographystyle{plain}
\bibliography{degenerations}                           
\end{document}